\documentclass[12pt]{amsart}
\usepackage{amsfonts,amssymb,a4wide,url,mathscinet,mathdots}
\usepackage[all]{xypic}


\numberwithin{equation}{section}

\newcommand{\A}{\mathbb{A}}
\newcommand{\C}{\mathbb{C}}
\newcommand{\R}{\mathbb{R}}
\newcommand{\Q}{\mathbb{Q}}
\newcommand{\Z}{\mathbb{Z}}
\newcommand{\bs}{\backslash}
\newcommand{\Whit}{\mathcal{W}}
\newcommand{\WhitM}{\mathcal{W}^{\psi}}
\newcommand{\ind}{\operatorname{ind}}
\newcommand{\Ind}{\operatorname{Ind}}
\newcommand{\whitform}{A^\psi}
\newcommand{\trivchar}{\mathbf{1}}
\newcommand{\rest}{\big|}
\newcommand{\Irr}{\operatorname{Irr}}
\newcommand{\Tau}{\mathcal{T}}
\newcommand{\des}{\mathcal{D}}
\newcommand{\msqr}{\operatorname{msqr}}
\newcommand{\cusp}{\operatorname{cusp}}
\newcommand{\sqr}{\operatorname{sqr}}
\newcommand{\temp}{\operatorname{temp}}
\newcommand{\unit}{\operatorname{unit}}
\newcommand{\reg}{\operatorname{reg}}
\newcommand{\gen}{\operatorname{gen}}
\newcommand{\genpsi}[1]{{#1}\text{-}\operatorname{gen}}
\newcommand{\SL}{\operatorname{SL}}
\newcommand{\GL}{\operatorname{GL}}
\newcommand{\SO}{\operatorname{SO}}
\newcommand{\OO}{\operatorname{O}}
\newcommand{\Sp}{\operatorname{Sp}}
\newcommand{\GSp}{\operatorname{GSp}}
\newcommand{\Mp}{\operatorname{Mp}}
\newcommand{\Ad}{\operatorname{Ad}}
\newcommand{\Artin}{\operatorname{Artin}}
\newcommand{\ari}{\operatorname{ar}}
\newcommand{\an}{}

\newcommand{\Hom}{\operatorname{Hom}}

\newcommand{\mira}{\mathcal{P}}
\newcommand{\modulus}{\delta}
\newcommand{\abs}[1]{\left|{#1}\right|}
\newcommand{\wgt}{\nu}
\newcommand{\diag}{\operatorname{diag}}
\newcommand{\central}{\omega}
\newcommand{\sprod}[2]{\left\langle#1,#2\right\rangle}  
\newcommand{\sm}[4]{{\bigl(\begin{smallmatrix}{#1}&{#2}\\{#3}&{#4}\end{smallmatrix}\bigr)}}
\newcommand{\Thetaone}{\overrightarrow{\Theta} \!}                                  
\newcommand{\Thetatwo}{\overleftarrow{\Theta} \!}                                   
\newcommand{\thetaone}{\overrightarrow{\theta} \! \!}                               
\newcommand{\thetatwo}{\overleftarrow{\theta} \! \!}                                
\newcommand{\JSlift}{\mathcal{L}}                                                

\newcommand{\Lie}{\operatorname{Lie}}
\newcommand{\Sym}{\operatorname{Sym}}
\newcommand{\eps}{\varepsilon}

\newcommand{\WD}{\mathit{WD}}
\newcommand{\DL}{\operatorname{DL}}

\newcommand{\bes}{{\mathcal B}}

\newcommand{\smth}{\operatorname{sm}}

\newcommand{\OG}{\SO(2n+1)}
\newcommand{\OP}{P^{O}}
\newcommand{\OM}{M^{O}}
\newcommand{\OU}{U^{O}}
\newcommand{\olevi}{\varrho^{O}}
\newcommand{\Ncirc}{N_{\GL_n}^{\circ}}
\newcommand{\ON}{N^{O}}
\newcommand{\inj}{\lambda}

\newcommand{\refcno}{}

\newtheorem{theorem}{Theorem}[section]
\newtheorem{lemma}[theorem]{Lemma}
\newtheorem{proposition}[theorem]{Proposition}
\newtheorem{corollary}[theorem]{Corollary}
\theoremstyle{remark}
\newtheorem{remark}[theorem]{Remark}

\begin{document}

\title[Formal degrees]{On the formal degrees of square-integrable representations of odd special orthogonal and metaplectic groups}
\author{Atsushi Ichino}
\address{Department of Mathematics, Kyoto University, Kyoto 606-8502, Japan}
\email{ichino@math.kyoto-u.ac.jp}
\author{Erez Lapid}
\address{Department of Mathematics, Weizmann Institute of Science, Rehovot 76100, Israel}
\email{erez.m.lapid@gmail.com}
\author{Zhengyu Mao}
\address{Department of Mathematics and Computer Science, Rutgers University, Newark, NJ 07102, USA}
\email{zmao@rutgers.edu}
\thanks{First named author partially supported by JSPS Grant-in-Aid for Scientific Research (B) 26287003 and by a grant from the Sumitomo Foundation}
\thanks{Second and third named authors partially supported by U.S.--Israel Binational Science Foundation Grant \# 057/2008}
\thanks{Second named author partially supported by a grant from the Minerva Stiftung}
\thanks{Third named author partially supported by NSF grant DMS 1000636 and by a fellowship from the Simons Foundation}
\date{\today}
\begin{abstract}
The formal degree conjecture relates the formal degree of an irreducible square-integrable representation of a reductive group over a local field to the special value of the adjoint $\gamma$-factor of its $L$-parameter.
In this paper, we prove the formal degree conjecture for odd special orthogonal and metaplectic groups in the generic case, which combined with Arthur's work on the local Langlands correspondence implies the conjecture in full generality.
\end{abstract}
\maketitle

\setcounter{tocdepth}{1}
\tableofcontents

\section{Introduction}

Let $F$ be a non-archimedean local field of characteristic $0$ and $G$ a connected reductive algebraic group over $F$.
We write $G = G(F)$ by abuse of notation.
Let $Z$ be the maximal $F$-split torus of the center of $G$.
Let $\pi$ be an irreducible square-integrable representation of $G$.
Namely, $\pi$ is an irreducible smooth representation of $G$ with unitary central character such that the absolute value
of any matrix coefficient of $\pi$ is square-integrable over $Z \backslash G$.
Let $\pi^{\vee}$ be the contragredient of $\pi$, and $[\cdot, \cdot]$ the standard pairing on $\pi \times \pi^{\vee}$.
Recall that the \emph{formal degree} $d_{\pi}$ of $\pi$ is the Haar measure on $Z\bs G$ for which
\[
 \int_{Z \bs G} [\pi(g) v_1, v_1'] [\pi(g^{-1}) v_2, v_2'] \, d_\pi g = [v_1, v_2'] [v_2, v_1']
\]
for $v_1, v_2 \in \pi$ and $v_1', v_2' \in \pi^{\vee}$.

The formal degree $d_{\pi}$ is a representation-theoretic invariant of $\pi$ introduced by Harish-Chandra.
It plays an important role in the harmonic analysis on $G$.
Fix a non-trivial character $\psi$ of $F$. One can attach to $\psi$ a Haar measure $d_\psi$ on $Z\bs G$ (see \cite{MR1458303, MR2425185}).
The formal degree conjecture \cite{MR2350057, MR2425185}, which is a far-reaching generalization of the Weyl dimension formula, relates the
constant $d_{\psi}/d_{\pi}$ to the adjoint $\gamma$-factor $\gamma(s, \pi, \Ad, \psi)$.
The latter is an arithmetic invariant of $\pi$ which in general is only defined via the (conjectural) local Langlands correspondence
for $G$ as an Artin factor. Henceforth we will write $\gamma^{\ari}(s, \pi, \Ad, \psi)$ to emphasize that.
In certain cases one can define instead, without assuming the local Langlands correspondence, an analytic factor $\gamma^{\an}(s, \pi, \Ad, \psi)$
via the theory of Rankin--Selberg integrals or the Langlands--Shahidi method.\footnote{These factors satisfy certain conditions which determine them
uniquely.}
Conjecturally,
\begin{equation} \label{an=ar}
\gamma^{\ari}(s, \pi, \Ad, \psi)=\gamma^{\an}(s, \pi, \Ad, \psi)
\end{equation}
so that one can phrase the formal degree conjecture in terms of $\gamma^{\an}(s, \pi, \Ad, \psi)$.
The equality \eqref{an=ar} is known in some cases, most importantly for Rankin--Selberg convolutions.

In this paper, we study the formal degree conjecture for the general linear group $\GL_n$, the metaplectic group $\Mp_n$,
i.e., the unique non-split double cover of the symplectic group $\Sp_n$, and the odd special orthogonal groups $\SO(2n+1)$.

In \S \ref{s: gl}, we consider the case of $\GL_n$.
The formal degree conjecture is already known in this case by using either an explicit formula of Silberger--Zink
(see \cite[Theorem 6.5]{MR2131140}, \cite[Theorem 3.1]{MR2350057}) or the Langlands--Shahidi method as in \cite[\S 4]{MR2350057}.
We give a new proof based on the Rankin--Selberg method (see Theorem \ref{thm: GLn}).
In this case the equality \eqref{an=ar} is a consequence of the local Langlands correspondence for $\GL_n$ \cite{MR1876802, MR1738446, MR3049932}.

In \S\ref{s: mp}, we consider the case of $\Mp_n$.
Strictly speaking, $\Mp_n$ is not an algebraic group, but the conjecture for $\Mp_n$ is formulated (slightly inaccurately) in \cite[\S 14]{MR3166215}.
The work of Jiang--Soudry \cite{MR1983781, MR2058617}, which is based on the descent method of Ginzburg--Rallis--Soudry \cite{MR1671452, MR1954940, MR2848523},
gives for a suitable choice of a non-generate character $\psi_{\tilde N}$, a one-to-one correspondence between
the set $\Irr_{\sqr,\genpsi{\psi_{\tilde N}}}\Mp_n$ of irreducible $\psi_{\tilde N}$-generic square-integrable representations of $\Mp_n$
and the set $\Irr_{\msqr}\GL_{2n}$ of irreducible representations of $\GL_{2n}$ parabolically induced from $\pi_1 \otimes \dots \otimes \pi_k$,
where $\pi_1,\dots,\pi_k$ are pairwise inequivalent irreducible square-integrable representations of
$\GL_{2n_1},\dots,\GL_{2n_k}$ respectively with $n_1+\dots+n_k=n$ and $L(0,\pi_i,\wedge^2)=\infty$ for all $i$.
We show that if $\tilde\pi\in\Irr_{\sqr,\genpsi{\psi_{\tilde N}}}\Mp_n$ is the descent of $\pi\in\Irr_{\msqr}\GL_{2n}$ then
\[
d_\psi=\abs{2}^n 2^k \gamma^{\an}(1,\pi,\Sym^2,\psi)d_{\tilde\pi}=\abs{2}^n 2^k \gamma^{\ari}(1,\pi,\Sym^2,\psi)d_{\tilde\pi}
\]
(see Theorem \ref{thm: Mpn}).
Here $\gamma^{\an}(s,\pi,\Sym^2,\psi)$ is the $\gamma$-factor defined by Shahidi \cite{MR1070599}.
The proof is based on the Main Identity of the second and third named authors \cite{1404.2905}.

In \S\ref{s: so}, we consider the split odd special orthogonal group $\SO(2n+1)$.
The theta correspondence gives a bijection between $\Irr_{\sqr,\genpsi{\psi_{\tilde N}}}\Mp_n$
and $\Irr_{\sqr,\gen}\SO(2n+1)$. (Since $\SO(2n+1)$ is adjoint, we do not need to specify the generic character.)
Using the above result for $\Mp_n$ and the result of \cite{MR3166215} we conclude that for any $\sigma\in\Irr_{\sqr,\gen}\SO(2n+1)$
\begin{equation} \label{eq: so2n+1}
 d_\psi = 2^k \gamma^{\an}(1, \pi, \Sym^2, \psi)d_\sigma=2^k \gamma^{\ari}(1, \pi, \Sym^2, \psi)d_\sigma
\end{equation}
where $\pi\in \Irr_{\msqr} \GL_{2n}$ is the Jiang--Soudry lift of $\sigma$ \cite{MR2058617}.
Actually, some of the results in \S\ref{s: so} are necessary for \S \ref{s: mp}, as are results about
the theta correspondence due to Gan--Savin \cite{MR2999299}.

Under the local Langlands correspondence for $\SO(2n+1)$, which would follow from Arthur's work \cite{MR3135650} once its
prerequisites are established,
we get the formal degree conjecture for $\SO(2n+1)$ (as well as for $\Mp_n$) in full generality in \S\ref{sec: odd general}.
Indeed, given a square-integrable $L$-packet $\mathfrak{P}$, the relation \eqref{eq: so2n+1} is the required relation
for the generic member of $\mathfrak{P}$ and as explained in \cite[\S9]{MR1070599}, the endoscopic character relations show that
the formal degree is constant on $\mathfrak{P}$.
Similarly, a similar statement for non-split $\SO(2n+1)$ would follow from results announced by Arthur.

We remark that in \cite{MR2350057} a different method using stable endoscopy was used to study the formal degrees of stable square-integrable
representations for odd unitary groups. However, it seems non-trivial to use this method to study non-stable representations.

We mention some prior results confirming the formal degree conjecture for inner forms of $\GL_n$ and $\SL_n$ \cite{MR2350057},
unipotent discrete series representations of adjoint unramified groups \cite{MR1748271, MR2891878, 1310.7790},
depth-zero supercuspidal representations of pure inner forms of unramified groups \cite{MR2480618}, \cite[\S 3.5]{MR2350057},
and certain positive-depth supercuspidal representations of tamely ramified groups \cite{MR2730575, MR3164986, 1209.1720}.
The formal degree conjecture is also known for $\operatorname{U}(3)$, $\GSp_2$, $\Sp_2$ \cite{MR3166215}, and $\Mp_1$ \cite{MR2876383}.

The proof of our results is based on a comparison of two inner products on the Whittaker model.
One is the integration over $ZN\bs G$ (where $N$ is a maximal unipotent subgroup of $G$) which makes sense for any irreducible
generic square-integrable representation.
The other comes from Rankin--Selberg integrals and it involves the group $\GL_n$ in the background.
By a formal computation, the constant of proportionality is precisely the formal degree up to a sign.
On the other hand, this constant is the local $\gamma$-factors up to precise factors which in the case of classical groups and the metaplectic group
were worked out by Kaplan in \cite{Kaplan}.

In Appendix \ref{app: globalization} we provide a globalization result, based on a result of Sakellaridis--Venkatesh \cite{1203.0039}
and bounds towards the Ramanujan conjecture on $\GL_n$ by Rudnick--Luo--Sarnak \cite{MR1703764} (extended to the ramified case independently by
Bergeron--Clozel \cite{MR2245761} and M\"uller--Speh \cite{MR2053600}). As usual, it enables us to use global methods for local results.

Finally, in Appendix \ref{app: real} we consider the real case.
As we mentioned above, in the $p$-adic case, our result on the formal degree conjecture for $\Mp_n$ is a consequence of the Main Identity \cite{1404.2905}.
Conversely, in the real case, we will deduce the Main Identity (in the square-integrable case) from the formal degree conjecture, which is a reformulation of Harish-Chandra's formula for formal degrees.

\subsection*{Acknowledgement}

We would like to thank Joseph Bernstein, Wee Teck Gan, Eyal Kaplan, Gordan Savin, David Soudry and Marko Tadi\'c for useful discussions.

\subsection*{Notation}

Let $F$ be a non-archimedean local field of characteristic $0$ and of residual characteristic $p$.
We denote by $\mathcal{O}$ the ring of integers of $F$, by $q$ the cardinality of the residue field of $F$, and by $\abs{\, \cdot \,}$
the absolute value on $F$.
Let $W_F$ be the Weil group of $F$ and $\WD_F = W_F \times \SL(2,\C)$ the Weil--Deligne group of $F$.
Fix a non-trivial character $\psi$ of $F$.
For a linear algebraic group $G$ over $F$, we identify $G$ with the group of its $F$-rational points $G(F)$.
We denote by $e$ the identity element of $G$ and by $\modulus_G$ the modulus character of $G$.
If $G$ is connected and reductive, we denote by $Z$ the maximal $F$-split torus of the center of $G$.
For a positive integer $m$, we denote by $I_m$ the identity matrix in $\GL_m$.
Set
\[
 w_m = \left(\begin{smallmatrix} & & 1 \\ & \iddots & \\ 1 & & \end{smallmatrix}\right) \in \GL_m.
\]
Set $x^* = w_m {}^t x^{-1} w_m$ for $x \in \GL_m$.
We denote by $\mathrm{M}_m$ the $F$-vector space of $m \times m$ matrices.
Set $\mathfrak{s}_m = \{ x \in \mathrm{M}_m : \breve{x} = x \}$, where $\breve{x} = w_m {}^t x w_m$.

Given a connected algebraic subgroup $G$ of $\GL_m$ defined over $F$ we can endow $G$ with a Haar measure $d_\psi=d_\psi^G$
by considering the lattice of integral matrices in $\Lie G$ and using it to determine a gauge form (up to an element of $\mathcal{O}^*$)
which in turn (together with the self-dual Haar measure on $F$ with respect to $\psi$) gives rise to a Haar measure on $G$.
Note that this coincides with the Haar measure defined in \cite{MR1458303, MR2425185}
(using a Chevalley basis) for the symplectic group, but not for the orthogonal group
(with their canonical linear representations).
If $H$ is a subgroup of $G$ we take the quotient ``measure'' $d_\psi^{H \backslash G} = \frac{d_\psi^G}{d_\psi^H}$ on $H \backslash G$.
(Of course, strictly speaking this is a measure only when $\modulus_G\rest_H=\modulus_H$.)
When the choice of Haar measure is unimportant (e.g., for convergence estimates) we omit the subscript $\psi$ from the notation.

We denote by $\Irr G$ the set of equivalence classes of irreducible smooth representations of $G$.
From now on, we do not distinguish representations and their equivalence classes.
Let
\[
 \Irr_{\sqr} G \subset \Irr_{\temp} G \subset \Irr_{\unit} G \subset \Irr G
\]
be the chain of subsets consisting of irreducible square-integrable (resp.~tempered, unitarizable) representations.
We denote by $\Irr_{\cusp} G$ the subset of irreducible supercuspidal representations of $G$.
If $G$ is quasi-split over $F$, $N$ is a maximal unipotent subgroup of $G$ defined over $F$ and $\psi_N$ is a non-degenerate
character of $N$, we denote by $\Irr_{\genpsi{\psi_N}} G$ the subset of irreducible $\psi_N$-generic representations of $G$.
If $G/Z$ is adjoint, then $\psi_N$ is unique up to conjugacy, so we omit $\psi_N$ from the notation.
If $p_1,p_2,\dots$ are properties of representations, we write $\Irr_{p_1,p_2,\dots}G=\Irr_{p_1}G\cap\Irr_{p_2}G\cap\cdots$.
If $\tilde G = \Mp_n$ we will only consider genuine representations.

Let $\pi\in\Irr\GL_m$ and let $\phi$ be the $m$-dimensional representation of $\WD_F$ corresponding to $\pi$ under the local
Langlands correspondence \cite{MR1876802, MR1738446, MR3049932}.
For any algebraic representation $r$ of $\GL_m(\C)$, write
\[
 \gamma^{\ari}(s, \pi, r, \psi) = \eps^{\DL}(s, r\circ\phi, \psi) \cdot \frac{L^{\Artin}(1-s, r^{\vee}\circ\phi)}{L^{\Artin}(s, r\circ\phi)}
\]
where $r^{\vee}$ is the contragredient of $r$ and $L^{\Artin}$ and $\eps^{\DL}$ are the factors attached to a representation of $\WD_F$ by
Artin and Deligne--Langlands respectively (see \cite{MR546607}).
If $r$ is the standard representation, we suppress $r$ from the notation.
The properties of local Langlands correspondence for $\GL_m$ guarantee that
\[
\gamma^{\ari}(s,\pi_1\times\pi_2,\psi)=\gamma^{\an}(s,\pi_1\times\pi_2,\psi)
\]
for any $\pi_i\in\Irr\GL_{m_i}$, $i=1,2$ where the left-hand side is the $\gamma$-factor attached to the tensor product representation of
$\GL_{m_1}(\C) \times \GL_{m_2}(\C)$ and the right-hand is the $\gamma$-factor defined by Jacquet--Piatetski-Shapiro--Shalika \cite{MR701565}
(or Shahidi \cite{MR1070599}; they coincide, cf.~\cite{MR729755}).

For non-negative real-valued functions $a$ and $b$ defined on a set $X$ we write $a(x) \ll b(x)$ if there exists a constant $c>0$ such that
$a(x) \le c b(x)$ for all $x\in X$. If $c$ depends on a parameter $d$, we write $a(x) \ll_d b(x)$.

\section{General linear groups} \label{s: gl}

Let $n$ be a positive integer and $G = \GL_n$ the general linear group of rank $n$.
Let $N$ be the maximal unipotent subgroup of $G$ consisting of upper unitriangular matrices
and $B=B_n$ the normalizer of $N$, i.e., the group of upper triangular matrices in $G$.
Let $\mira$ be the mirabolic subgroup of $G$ consisting of matrices whose last row is $(0,\ldots,0,1)$.
We denote by $Z$ the center of $G$.

Let $\pi\in\Irr_{\gen}G$.
We realize $\pi$ on its Whittaker model $\Whit^{\psi}(\pi)$ with respect to the non-degenerate character $\psi_N$ of $N$ given by
$\psi_N(u) = \psi(u_{1,2} + \cdots + u_{n-1,n})$. Assume that $\pi$ is unitarizable.
We define a non-degenerate $G$-invariant bilinear form $[\cdot, \cdot]_\pi$ on $\Whit^{\psi}(\pi) \times \Whit^{\psi^{-1}}(\pi^\vee)$ by
\[
[W,W']_\pi=\int_{N\bs\mira}W(p)W'(p)\, d_\psi p,
\]
where the integral converges absolutely (see \cite{MR748505}).
Assume that $\pi$ is square-integrable.
By \cite[Lemma 4.4]{LMao5} we have
\begin{equation} \label{FW}
\int_N[\pi(u)W,W']_\pi\psi_N(u)^{-1}\, d_\psi u=W(e)W'(e)
\end{equation}
for $W\in\Whit^{\psi}(\pi)$ and $W'\in\Whit^{\psi^{-1}}(\pi^\vee)$,
where the integral converges absolutely (see \cite[Proposition II.4.5]{MR1989693}).
We can also define another non-degenerate $G$-invariant bilinear form $(\cdot, \cdot)_\pi$ on $\Whit^{\psi}(\pi) \times \Whit^{\psi^{-1}}(\pi^\vee)$ by
\[
(W,W')_\pi=\int_{ZN\bs G}W(g)W'(g)\, d_\psi g,
\]
where the integral converges absolutely (see \cite{MR2495561}).
By \cite[Lemma A.1]{MR2930996}\footnote{The result in [loc.~cit.] is stated with respect to the Haar measures
\[
 \prod_{j=1}^n L(j,\trivchar_{F^*}) \cdot d_\psi g, \qquad d_\psi u, \qquad
 \prod_{j=1}^{n-1} L(j,\trivchar_{F^*}) \cdot d_\psi p, \qquad L(1,\trivchar_{F^*})\cdot d_\psi z
\]
on $G$, $N$, $\mira$, $Z$ respectively.},
which was suggested by Jacquet, we have
\begin{equation} \label{prop}
(W,W')_\pi=n\central_\pi(-1)^{n-1}\gamma^{\an}(1,\pi,\Ad,\psi)[W,W']_\pi
\end{equation}
for $W\in\Whit^{\psi}(\pi)$ and $W'\in\Whit^{\psi^{-1}}(\pi^\vee)$, where $\central_\pi$ is the central character of $\pi$ and
$\Ad$ is the adjoint representation of $\GL_n(\C)$ on $\mathfrak{sl}_n(\C)$. Here
\[
\gamma^{\an}(s,\pi,\Ad,\psi)=\gamma^{\an}(s,\pi\times\pi^\vee,\psi)/\gamma^{\an}(s,\trivchar_{F^*},\psi)
\]
where we recall that the numerator is the $\gamma$-factor defined by Jacquet--Piatetski-Shapiro--Shalika \cite{MR701565}
and the denominator is Tate's $\gamma$-factor.
Of course, by the local Langlands correspondence for $\GL_n$ we have
\begin{equation} \label{eq: Adanarequality}
\gamma^{\an}(s,\pi,\Ad,\psi)=\gamma^{\ari}(s,\pi,\Ad,\psi).
\end{equation}

\begin{theorem} \label{thm: GLn}
Let $\pi \in \Irr_{\sqr} \GL_n$ and let $d_{\pi}$ be its formal degree.
Then
\[
d_\psi=n\central_\pi(-1)^{n-1}\gamma^{\an}(1,\pi,\Ad,\psi)d_\pi=
n\central_\pi(-1)^{n-1}\gamma^{\ari}(1,\pi,\Ad,\psi)d_\pi.
\]
\end{theorem}

\begin{proof}
By \eqref{eq: Adanarequality} it is enough to prove the first equality.
Recall that $d_\pi$ is defined by the relation
\[
\int_{Z\bs G}[\pi(g)W_1,W_1']_\pi[\pi(g^{-1})W_2,W_2']_\pi\, d_\pi g=[W_1,W_2']_\pi[W_2,W_1']_\pi
\]
for $W_1,W_2\in\Whit^{\psi}(\pi)$ and $W_1',W_2'\in\Whit^{\psi^{-1}}(\pi^\vee)$.
Let us compute
\[
\int_{Z\bs G}[\pi(g)W_1,W_1']_\pi[\pi(g^{-1})W_2,W_2']_\pi\, d_\psi g,
\]
i.e.,
\[
\int_{Z\bs G} \left(\int_{N\bs\mira}W_1(pg)W_1'(p)\, d_\psi p \right) [\pi(g^{-1})W_2,W_2']_\pi\, d_\psi g.
\]
We will soon see that the double integral converges absolutely.
Therefore we can interchange the order of integration to get
\[
\int_{N\bs\mira}\int_{Z\bs G}W_1(pg)W_1'(p)[\pi(g^{-1})W_2,W_2']_\pi\, d_\psi g\, d_\psi p.
\]
Changing the variable $g\mapsto p^{-1}g$, we get
\begin{align*}
 & \int_{N\bs\mira}\int_{Z\bs G}W_1(g)W_1'(p)[\pi(g^{-1}p)W_2,W_2']_\pi\, d_\psi g\, d_\psi p \\
 & = \int_{N\bs\mira}\int_{ZN\bs G}\int_N\psi_N(u)W_1(g)W_1'(p)[\pi(g^{-1}u^{-1}p)W_2,W_2']_\pi\, d_\psi u\, d_\psi g\, d_\psi p \\
 & = \int_{N\bs\mira}\int_{ZN\bs G}\int_N W_1(g)W_1'(p)[\pi(up)W_2,\pi^{\vee}(g)W_2']_\pi\psi_N(u)^{-1} \, d_\psi u\, d_\psi g\, d_\psi p.
\end{align*}
By \eqref{FW}, this is equal to
\[
\int_{N\bs\mira}\int_{ZN\bs G}W_1(g)W_1'(p)W_2(p)W_2'(g)\, d_\psi g\, d_\psi p=
(W_1,W_2')_\pi[W_2,W_1']_\pi.
\]
Now the theorem follows from \eqref{prop}.

To justify the manipulation, we show the convergence of the triple integral
\[
\int_{N\bs\mira}\int_{ZN\bs G}\int_N\abs{W_1(g)W_1'(p)[\pi(g^{-1}u^{-1}p)W_2,W_2']_\pi} du\, dg\, dp.
\]
Let $K=K_n = \GL_n(\mathcal{O})$ be the standard maximal compact subgroup of $G$.
We embed $\GL_{n-1}$ into $\GL_n$ by $g\mapsto\sm g{}{}1$.
Let $T$ be the maximal torus of $\GL_{n-1}$ consisting of diagonal matrices.
Using the Iwasawa decomposition, we write the above integral as
\begin{multline*}
\int_{K_{n-1}}\int_{K_n}\int_T\int_T\int_N
\abs{W_1(t_1k_1)W_1'(t_2k_2)[\pi(k_1^{-1}t_1^{-1}ut_2k_2)W_2,W_2']_\pi} \\
 \times \modulus_{B_n}(t_1)^{-1}\modulus_{B_{n-1}}(t_2)^{-1}\, du\, dt_1\, dt_2\, dk_1\, dk_2.
\end{multline*}
Changing the variable $u \mapsto t_1 u t_1^{-1}$, we get
\begin{multline*}
\int_{K_{n-1}}\int_{K_n}\int_T\int_T\int_N\abs{W_1(t_1k_1)W_1'(t_2k_2)[\pi(k_1^{-1}ut_1^{-1}t_2k_2)W_2,W_2']_\pi}\\
\modulus_{B_{n-1}}(t_2)^{-1}\, du\, dt_1\, dt_2\, dk_1\, dk_2.
\end{multline*}
Since $\pi$ is square-integrable, there exists a function $f$ in the Harish-Chandra Schwartz space of $G$ such that
\[
 [\pi(g)W_2,W_2']_\pi = \int_Z f(zg) \omega_{\pi}(z)^{-1} \, dz
\]
(see \cite[Th\'eor\`eme VIII.4.2]{MR1989693}).
By definition we have $\abs{f(g)} \ll_d \Xi^G(g)\sigma(g)^{-d}$ for any $d >0$.
Here $\Xi^G$ is the Harish-Chandra standard spherical function on $G$ and $\sigma(g)=\max(1,\log\abs{g_{i,j}},\log\abs{(g^{-1})_{i,j}})$
(cf.~\cite[p.~250 and 242]{MR1989693}). Hence we can bound the above integral by
\begin{multline*}
\int_{K_{n-1}}\int_{K_n}\int_T\int_T\int_Z\int_N\abs{W_1(t_1k_1)W_1'(t_2k_2)} \\
\times \Xi^G(uzt_1^{-1}t_2)\sigma(uzt_1^{-1}t_2)^{-d}\modulus_{B_{n-1}}(t_2)^{-1}\, du \, dz \, dt_1\, dt_2\, dk_1\, dk_2.
\end{multline*}
By \cite[Proposition II.4.5]{MR1989693}, for any $d' >0$, there exists $d>0$ such that
\[
\int_N\Xi^G(uzt)\sigma(uzt)^{-d}\, du \ll \modulus_{B_n}(t)^{\frac12}\sigma(zt)^{-d'}
\]
for all $z \in Z$ and $t \in T$. Thus, the above is bounded by a constant multiple of
\begin{multline*}
\int_{K_{n-1}}\int_T\abs{W_1'(t_2k_2)}
\modulus_{B_{n-1}}(t_2)^{-1}\modulus_{B_n}(t_2)^{\frac12}\\
\times\int_Z\int_{K_n}\int_T\abs{W_1(t_1k_1)}
\modulus_{B_n}(t_1)^{-\frac12}\sigma(zt_1^{-1}t_2)^{-d'}\, dt_1\, dk_1\, dz \, dt_2\, dk_2.
\end{multline*}
By the Cauchy--Schwarz inequality, the integral over $t_1,k_1$ is bounded by the square-root of
\begin{multline*}
\int_{K_n}\int_T\abs{W_1(t_1k_1)}^2\modulus_{B_n}(t_1)^{-1}\, dt_1\, dk_1\times
\int_{K_n}\int_T\sigma(zt_1^{-1}t_2)^{-2d'}\, dt_1\, dk_1\\=
\int_{ZN\bs G}\abs{W_1(g)}^2\, dg\times\int_T\sigma(zt_1^{-1})^{-2d'}\, dt_1.
\end{multline*}
Since
\[
\int_Z\left(\int_T\sigma(zt_1^{-1})^{-2d'}\, dt_1\right)^{\frac12}\, dz<\infty
\]
for $d'\gg1$ and $\modulus_{B_n}\rest_{B_{n-1}}=\modulus_{B_{n-1}}\cdot\abs{\det}$, we reduce the convergence to that of
\[
\int_{K_{n-1}}\int_T\abs{W_1'(t_2k_2)}\modulus_{B_n}(t_2)^{-\frac12}\abs{\det t_2}\, dt_2\, dk_2.
\]
This follows immediately from standard bounds on the Whittaker function (e.g., \cite[Lemma 2.1]{LMao4}).
\end{proof}

\section{Metaplectic groups}
\label{s: mp}

We follow the conventions of \cite{1401.0198, 1404.2905}.
Let $n$ be a positive integer and let $$G' = \Sp_n = \{g \in \GL_{2n} : {}^tg J_n g = J_n \}$$ be the symplectic group of rank $n$, where $J_n = \sm{}{w_n}{-w_n}{}$.
Let $N'$ be the maximal unipotent subgroup of $G'$ consisting of upper unitriangular matrices.
Let $\tilde{G} = \Mp_n$ be the metaplectic group of rank $n$, i.e., the unique non-split double cover of $G'$.
We regard $\tilde{G}$ as the set $G' \times \{ \pm 1 \}$ with multiplication law determined by Ranga Rao's $2$-cocycle.
We write $\tilde{g} = (g,1) \in \tilde{G}$ for $g \in G'$.
We identify $N'$ with its image in $\tilde{G}$ under the homomorphism $u \mapsto \tilde{u}$.
Let $\tilde N$ be the preimage of $N'$ in $\tilde G$.

Let $G = \Sp_{2n}$ be the symplectic group of rank $2n$ and $\eta:G' \hookrightarrow G$ the embedding given by $\eta(g) = \diag(I_n,g,I_n)$.
Let $N$ be the maximal unipotent subgroup of $G$ consisting of upper unitriangular matrices.
Let $V$ be the unipotent radical of the standard parabolic subgroup of $G$ with Levi component $\GL_1 \times \cdots \times \GL_1 \times \Sp_n$,
so that $N = V \rtimes \eta(N')$.
Let $P = M\ltimes U$ be the Siegel parabolic subgroup of $G$ with Levi component
$M = \{ \varrho(m) : m \in \GL_{2n} \}$ and unipotent radical $U = \{ \ell(x) : x \in \mathfrak{s}_{2n} \}$,
where $\varrho(m) = \diag(m, m^*)$ and $\ell(x) = \sm{I_{2n}}{x}{}{I_{2n}}$.
Let $K = G \cap \GL_{4n}(\mathcal{O})$ be the standard maximal compact subgroup of $G$.

Let $\pi\in\Irr_{\gen}\GL_{2n}$.
We regard $\pi$ as a representation of $M$ via $\varrho$ and realize it on the Whittaker model $\WhitM(\pi)$ with respect to the
non-degenerate character $\psi_{N_M}$ of $N_M = N \cap M$ given by $\psi_{N_M}(u)=\psi(u_{1,2}+ \cdots+u_{2n-1,2n})$.
Let $\Ind(\WhitM(\pi))$ be the space of left $U$-invariant smooth functions $W$ on $G$ such that the function
$m\mapsto\modulus_P(m)^{-\frac12} W(m g)$ on $M$ belongs to $\WhitM(\pi)$ for all $g\in G$.
For $W \in \Ind(\WhitM(\pi))$ and $s \in \C$, we define a function $W_s$ on $G$ by $W_s(g) = W(g) \abs{\det m}^s$,
where we choose $u \in U$, $m \in \GL_{2n}$, $k \in K$ such that $g = u \varrho(m) k$.
Let $\Ind(\WhitM(\pi), s)$ be the representation of $G$ on $\Ind(\WhitM(\pi))$ given by $(\Ind(\WhitM(\pi), s)(g)W)_s = W_s(\,\cdot \, g)$.
We define an intertwining operator
\[
M(s): \Ind(\WhitM(\pi), s) \longrightarrow \Ind(\WhitM(\pi^\vee), -s)
\]
by the meromorphic continuation of the integral
\[
 (M(s)W)_{-s}(g) = \int_U W_s \left( \varrho(\mathfrak{t})
 \sm{}{I_{2n}}{-I_{2n}}{} u g \right) d_\psi u,
\]
where $\mathfrak{t} = \diag(1,-1,1,\ldots,-1) \in \GL_{2n}$.
Recall that $M(s)$ converges absolutely and is holomorphic for $\Re s>0$ if $\pi$ is tempered.

Let $\omega_{\psi^{-1}}$ be the (suitably extended) Weil representation of $V \rtimes \tilde{G}$ with respect to $\psi^{-1}$ on the space
$\mathcal{S}(F^n)$ of Schwartz--Bruhat functions on $F^n$ (see \cite[\S 2.4]{1401.0198}).
Let
\[
\whitform(W,\Phi,\tilde g,s) =\int_{V_\gamma\bs V} W_s(\gamma v\eta(g))\omega_{\psi^{-1}}(v \tilde g)\Phi(\xi_n)\,d_\psi v
\]
for $W \in \Ind(\WhitM(\pi))$, $\Phi \in \mathcal{S}(F^n)$, $g \in G'$, and $s \in \C$, where
\[
 \gamma = \left(\begin{smallmatrix} & I_n & & \\ & & & I_n \\ -I_n & & & \\ & & I_n & \end{smallmatrix}\right)=
 \begin{pmatrix}&I_{2n}\\-I_{2n}&\end{pmatrix}
 \left(\begin{smallmatrix}I_n&&&\\&&-I_n&\\&I_n&&\\&&&I_n\end{smallmatrix}\right),
\]
$V_{\gamma} = V \cap \gamma^{-1} N \gamma$, and $\xi_n = (0, \ldots, 0, 1) \in F^n$.
By \cite[Lemma 4.5]{1401.0198}, this integral converges absolutely and defines an entire function in $s$, and
$\whitform(W,\Phi, \tilde u \tilde g,s) = \psi_{\tilde N}(\tilde u) \whitform(W,\Phi, \tilde{g},s)$ for $\tilde u \in \tilde N$,
where $\psi_{\tilde N}$ is the genuine character of $\tilde N$ whose restriction to $N'$ is the non-degenerate character
$$\psi_{\tilde N}(u) = \psi(u_{1,2} + \cdots + u_{n-1,n} - \tfrac{1}{2}u_{n,n+1}).$$

Consider now the local zeta integrals and local factors of Ginzburg--Rallis--Soudry \cite{MR1675971, MR1671452}.
Let $\tilde\sigma\in\Irr_{\genpsi{\psi_{\tilde N}^{-1}}}\tilde G$ with Whittaker model $\Whit^{\psi^{-1}}(\tilde{\sigma})$ and let $\pi \in \Irr_{\gen} \GL_{2n}$.
For any $\tilde{W} \in \Whit^{\psi^{-1}}(\tilde{\sigma})$, $W' \in \Ind(\Whit^{\psi}(\pi))$, $\Phi \in \mathcal{S}(F^n)$, and $s \in \C$ let
\[
\tilde{J}(\tilde W,W',\Phi,s)=\int_{N' \bs G'}\tilde W(\tilde{g})\whitform(W',\Phi,\tilde g,s)\, d_\psi g.
\]
The integral converges absolutely for $\Re s\gg0$ and is a rational function in $q^{-s}$.
Note that if $\tilde\sigma$ is square-integrable and $\pi$ is tempered then by \cite[Lemma 4.12]{1401.0198}
this integral converges absolutely for $\Re s \ge -\frac12$.
By \cite[Proposition 6.6]{MR1675971}, there exist $\tilde W$, $W'$, $\Phi$ such that $\tilde{J}(\tilde W,W',\Phi,s) = 1$ for all $s \in \C$.
Moreover, we have a local functional equation
\begin{equation} \label{fe}
 \tilde{J}(\tilde W,M(s)W',\Phi,-s)
 = \omega_\pi((-1)^n 2) \abs{2}^{2ns}\frac{\gamma^{\an}(s+\tfrac12, \tilde{\sigma} \times \pi, \psi)}
 {\gamma^{\an}(s, \pi, \psi) \gamma^{\an}(2s, \pi, \wedge^2, \psi)}
 \cdot \tilde{J}(\tilde W,W',\Phi,s)
\end{equation}
which defines the $\gamma$-factor $\gamma^{\an}(s, \tilde\sigma\times \pi, \psi)$ (see \cite{Kaplan}\footnote{In the definition of
$\Gamma(s, \pi \times \tau, \psi)$ in \cite[\S 4]{Kaplan}, $\omega_\pi(-1)$ in the metaplectic case should read the central sign
$\omega_\pi(-I_{2l}, 1) / \gamma_\psi((-1)^l)$ introduced by Gan--Savin \cite{MR2999299}.}).
Here $\omega_\pi$ is the central character of $\pi$ and the second factor in the denominator is the $\gamma$-factor defined by Shahidi \cite{MR1070599}.

Consider the subset $\Irr_{\msqr}\GL_{2n}$ of $\Irr_{\gen}\GL_{2n}$ consisting of representations of the form
\begin{equation} \label{eq: pi}
\pi=\pi_1\times\dots\times\pi_k := \Ind(\pi_1 \otimes \cdots \otimes \pi_k),
\end{equation}
where $\pi_1,\dots,\pi_k$ are pairwise inequivalent irreducible square-integrable representations of
$\GL_{2n_1},\dots,\GL_{2n_k}$ respectively with $n_1+\dots+n_k=n$ and $L(0,\pi_i,\wedge^2)=\infty$ for all $i$.
It follows from \cite[Theorem 4.3]{MR2846403}, \cite{MR3008415} that any $\pi \in \Irr_{\msqr}\GL_{2n}$ has a trivial central character.
Also, for $\pi$ as above $\gamma^{\an}(s, \pi, \wedge^2, \psi)$ has a pole of order $k$ at $s=1$, since the $\pi_i$'s are distinct.
Therefore,
\begin{equation} \label{eq: resk}
 \lim_{s \rightarrow \frac12} \frac{\gamma^{\an}(s + \frac12, \pi, \wedge^2, \psi)}{\gamma^{\an}(2s, \pi, \wedge^2, \psi)} = 2^k.
\end{equation}

For $\pi \in \Irr_{\msqr}\GL_{2n}$, we define a representation $\des_{\psi}(\pi)$ of $\tilde G$ by right translation on the space spanned
by the genuine functions $A^{\psi}(M(\frac12)W, \Phi, \cdot, -\frac12)$ on $\tilde{G}$ for $W \in \Ind(\WhitM(\pi))$ and $\Phi \in \mathcal{S}(F^n)$.
By \cite[p.~860, Theorem]{MR1671452}, we have $\des_{\psi}(\pi) \ne 0$.

The following result will be proved in the next section.
\begin{theorem} \label{thm: bijection}
The map $\pi\mapsto\des_\psi(\pi)$ defines a bijection
\[
\des_\psi:\Irr_{\msqr}\GL_{2n}\longrightarrow\Irr_{\sqr,\genpsi{\psi_{\tilde N}}}\Mp_n.
\]
Moreover, if $\pi\in\Irr_{\msqr} \GL_{2n}$ and $\tilde\pi=\des_{\psi^{-1}}(\pi)$ then
\begin{equation} \label{eq: polek}
\gamma^{\an}(s, \tilde{\pi} \times \tau, \psi) = \gamma^{\an}(s, \pi\times \tau, \psi)
\end{equation}
for $\tau \in \Irr_{\gen} \GL_m$, $m \ge 1$.
In particular, if $\pi$ is of the form \eqref{eq: pi} then $\gamma^{\an}(s, \tilde{\pi} \times \pi, \psi)$ has a pole of order $k$ at $s=1$.
\end{theorem}

We now quote a corollary of the main result of \cite{1404.2905}.
\begin{theorem}[{\cite[Corollary 3.4]{1404.2905}}] \label{thm: main}
Let $\pi\in\Irr_{\msqr}\GL_{2n}$ and $\tilde\pi=\des_{\psi^{-1}}(\pi) \in \Irr_{\sqr,\genpsi{\psi_{\tilde N}^{-1}}}\Mp_n$. Then
\begin{equation}\label{eq: main}
\int_{N'}\tilde{J}(\tilde\pi(u) \tilde W, W',\Phi,\tfrac12)\psi_{\tilde N}(u)\,d_\psi u
=\eps^{\an}(\tfrac12,\pi,\psi) \tilde W(e) \whitform(M(\tfrac 12)W',\Phi,e, -\tfrac12)
\end{equation}
for $\tilde{W} \in \Whit^{\psi^{-1}}(\tilde{\pi})$, $W' \in \Ind(\Whit^{\psi}(\pi))$, and $\Phi \in \mathcal{S}(F^n)$.
\end{theorem}

Note that \cite[Corollary 3.4]{1404.2905} uses Theorem \ref{thm: bijection}. However, this does not entail a circular reasoning
since the proof of Theorem \ref{thm: bijection} in the next section will be independent of the results of \cite{1404.2905}.

For any $\pi\in\Irr_{\gen}\GL_m$ let $\gamma^{\an}(s,\pi,\Sym^2,\psi)$ be the $\gamma$-factor defined by Shahidi \cite{MR1070599}.
We have
\begin{equation}
\label{eq: sym2wedge2}
\gamma^{\an}(s,\pi\times\pi,\psi) = \gamma^{\an}(s,\pi,\Sym^2,\psi)\gamma^{\an}(s,\pi,\wedge^2,\psi).
\end{equation}
With Theorem \ref{thm: main}, we will be able to deduce a formula for the formal degree as follows:
\begin{theorem} \label{thm: Mpn}
Assume that $\pi\in\Irr_{\msqr}\GL_{2n}$ is of the form \eqref{eq: pi}.
Let $\tilde \pi = \des_{\psi^{-1}}(\pi)\in \Irr_{\sqr,\genpsi{\psi_{\tilde N}^{-1}}}\Mp_n$ and let $d_{\tilde \pi}$ be its formal degree
(as a Haar measure on $\Sp_n$). Then
\[
d_\psi^{\Sp_n}=\abs{2}^n2^k\gamma^{\an}(1,\pi,\Sym^2,\psi)d_{\tilde\pi}=\abs{2}^n2^k\gamma^{\ari}(1,\pi,\Sym^2,\psi)d_{\tilde\pi}.
\]
\end{theorem}

\begin{remark}
The formal degree conjecture for $\Mp_n$ was formulated in \cite[\S 14]{MR3166215} but the factor $\abs{2}^n$ was overlooked.
\end{remark}

\begin{proof}
We first remark that the first asserted equality implies the second one. Indeed, by a result of Henniart \cite{MR2595008} we have
\[
 \gamma^{\an}(s, \pi, \Sym^2, \psi) = \alpha \gamma^{\ari}(s, \pi, \Sym^2, \psi)
\]
for some root of unity $\alpha$ (a priori depending on $\pi$). On the other hand
$\gamma^{\ari}(1, \pi, \Sym^2, \psi)$ is a positive real number (see \cite[(36) and \S 8.6]{MR2730575}).
Therefore $\alpha>0$ and hence $\alpha=1$.

Fix $\pi\in\Irr_{\msqr}\GL_{2n}$ of the form \eqref{eq: pi} and let $\tilde \pi = \des_{\psi^{-1}}(\pi)$.
By Theorem \ref{thm: bijection} $\tilde\pi$ is square-integrable.
We define a non-degenerate $\tilde G$-invariant bilinear form $[\cdot, \cdot]_{\tilde\pi}$ on
$\Whit^{\psi^{-1}}(\tilde \pi) \times \Whit^{\psi}(\tilde \pi^{\vee})$ as follows.
By the functional equation \eqref{fe}, $\tilde{J}(\tilde W,W',\Phi,\tfrac12)$ is a non-zero constant multiple of the
absolutely convergent integral $\tilde{J}(\tilde W,M(\tfrac12)W',\Phi,-\tfrac12)$.
Indeed, both $\gamma^{\an}(s+\tfrac12, \tilde{\pi} \times \pi, \psi)$ and $\gamma^{\an}(2s, \pi, \wedge^2, \psi)$ have a pole of order $k$ at $s=\frac12$,
whereas $\gamma^{\an}(s, \pi, \psi)$ is holomorphic and non-zero at $s=\frac12$.
Hence, the functional $W'\otimes \Phi \mapsto \tilde{J}(\tilde W,W',\Phi,\tfrac12)$ on $\Ind(\Whit^{\psi}(\pi)) \otimes \mathcal{S}(F^n)$
factors through the map $W' \otimes \Phi \mapsto \whitform(M(\tfrac 12)W',\Phi,\cdot, -\tfrac12)$.
In particular, $\tilde \pi^\vee = \des_{\psi}(\pi)$.
Let
\[
[\tilde W, \tilde W']_{\tilde\pi}=\tilde{J}(\tilde W,W',\Phi,\tfrac12)
\]
for $\tilde{W} \in \Whit^{\psi^{-1}}(\tilde{\pi})$ and $\tilde{W}' = \whitform(M(\tfrac 12)W',\Phi,\cdot, -\tfrac12) \in \Whit^{\psi}(\tilde{\pi}^{\vee})$.
Thus \eqref{eq: main} becomes
\begin{equation} \label{eq: mainnew}
\int_{N'}[\tilde\pi(u) \tilde W, \tilde W']_{\tilde\pi}\psi_{\tilde N}(u)\,d_\psi u=\eps^{\an}(\tfrac12,\pi,\psi) \tilde W(e) \tilde{W}'(e).
\end{equation}

We define another non-degenerate $\tilde G$-invariant bilinear form $(\cdot, \cdot)_{\tilde\pi}$ on
$\Whit^{\psi^{-1}}(\tilde \pi) \times \Whit^{\psi}(\tilde \pi^{\vee})$ by
\[
(\tilde W,\tilde W')_{\tilde\pi}=\int_{N'\bs G'}\tilde{W}(\tilde{g})\tilde{W}'(\tilde{g})\, d_\psi g,
\]
where the integral converges absolutely (see \cite{MR2495561}).
Then by the definition of $(\cdot,\cdot)_{\tilde{\pi}}$, the functional equation \eqref{fe}, and \eqref{eq: resk},
we have
\begin{align*}
 & (\tilde W,\whitform(M(\tfrac 12)W',\Phi, \cdot,-\tfrac12))_{\tilde\pi}
  = \tilde{J}(\tilde W,M(\tfrac12)W',\Phi,-\tfrac12) \\
 & = \abs{2}^n\lim_{s\rightarrow\frac12}\frac{\gamma^{\an}(s+\frac12, \tilde \pi \times \pi, \psi)}{\gamma^{\an}(s, \pi, \psi)
 \gamma^{\an}(2s, \pi, \wedge^2, \psi)}
 \cdot \tilde{J}(\tilde W,W',\Phi,\tfrac12) \\
 & = \abs{2}^n2^k \cdot\lim_{s\rightarrow1}\frac{\gamma^{\an}(s, \tilde \pi \times \pi, \psi)}{\eps^{\an}(\frac12, \pi, \psi)
 \gamma^{\an}(s, \pi, \wedge^2, \psi)}
 \cdot [\tilde W,\whitform(M(\tfrac 12)W',\Phi,\cdot, -\tfrac12)]_{\tilde\pi}.
\end{align*}
Taking into account \eqref{eq: polek} and \eqref{eq: sym2wedge2} we get
\begin{equation} \label{eq: prop_Mp}
(\tilde W,\tilde W')_{\tilde\pi}=\abs{2}^n2^k \cdot \frac{\gamma^{\an}(1, \pi, \Sym^2, \psi)}{\eps^{\an}(\frac12, \pi, \psi)} \cdot [\tilde W,\tilde W']_{\tilde\pi}
\end{equation}
for $\tilde{W} \in \Whit^{\psi^{-1}}(\tilde{\pi})$ and $\tilde{W}' \in \Whit^{\psi}(\tilde{\pi}^{\vee})$.

Recall that the Haar measure $d_{\tilde \pi}$ is defined by the relation
\[
\int_{G'}[\tilde\pi(\tilde{g})\tilde W_1,\tilde W_1']_{\tilde\pi}[\tilde\pi(\tilde{g}^{-1})\tilde W_2,\tilde W_2']_{\tilde\pi}\, d_{\tilde\pi} g=
[\tilde W_1,\tilde W_2']_{\tilde\pi}[\tilde W_2,\tilde W_1']_{\tilde\pi}
\]
for $\tilde{W}_1, \tilde{W}_2 \in \Whit^{\psi^{-1}}(\tilde{\pi})$ and $\tilde{W}'_1, \tilde{W}'_2 \in \Whit^{\psi}(\tilde{\pi}^{\vee})$.
Assume that $\tilde W_1'=\whitform(M(\frac 12)W_1',\Phi,\cdot,-\frac12)$ with $W_1' \in \Ind(\Whit^{\psi}(\pi))$ and $\Phi \in \mathcal{S}(F^n)$.
As in the case of $\GL_n$ we compute
\[
\int_{G'}[\tilde\pi(\tilde{g})\tilde W_1,\tilde W_1']_{\tilde\pi}[\tilde\pi(\tilde{g}^{-1})\tilde W_2,\tilde W_2']_{\tilde\pi}\, d_\psi g,
\]
i.e.,
\[
 \int_{G'}\left( \int_{N'\bs G'}\tilde W_1(\tilde{x} \tilde{g})\whitform(W_1',\Phi,\tilde x,\tfrac12)\, d_\psi x \right)
 [\tilde\pi(\tilde{g}^{-1})\tilde W_2,\tilde W_2']_{\tilde\pi}\, d_\psi g.
\]
We will soon see that the double integral converges absolutely.
Therefore we can interchange the order of integration to get
\[
 \int_{N'\bs G'}\int_{G'} \tilde W_1(\tilde{x} \tilde{g})\whitform(W_1',\Phi,\tilde x,\tfrac12)
 [\tilde\pi(\tilde{g}^{-1})\tilde W_2,\tilde W_2']_{\tilde\pi}\, d_\psi g \, d_\psi x.
\]
Changing the variable $g\mapsto x^{-1}g$, we get
\begin{align*}
 &\int_{N'\bs G'}\int_{G'}\tilde W_1(\tilde{g})\whitform(W_1',\Phi,\tilde x,\tfrac12)
 [\tilde\pi(\tilde{g}^{-1}\tilde{x})\tilde W_2,\tilde W_2']_{\tilde\pi}\, d_\psi g\, d_\psi x \\
 &=\int_{N'\bs G'}\int_{N'\bs G'}\int_{N'}\psi_{\tilde N}(u)^{-1}\tilde W_1(\tilde{g})\whitform(W_1',\Phi,\tilde x,\tfrac12)
 [\tilde\pi(\tilde{g}^{-1}u^{-1}\tilde{x})\tilde W_2,\tilde W_2']_{\tilde\pi}\, d_\psi u\, d_\psi g\, d_\psi x \\
 &=\int_{N'\bs G'}\int_{N'\bs G'}\int_{N'}\tilde W_1(\tilde{g})\whitform(W_1',\Phi,\tilde x,\tfrac12)
 [\tilde\pi(u\tilde{x})\tilde W_2,\tilde\pi^{\vee}(\tilde{g})\tilde W_2']_{\tilde\pi}\psi_{\tilde N}(u)\, d_\psi u\, d_\psi g\, d_\psi x.
\end{align*}
By \eqref{eq: mainnew}, this is equal to
\begin{align*}
 & \eps^{\an}(\tfrac12, \pi, \psi) \int_{N'\bs G'}\int_{N'\bs G'}\tilde W_1(\tilde{g})\whitform(W_1',\Phi,\tilde x,\tfrac12)
 \tilde W_2(\tilde{x}) \tilde W_2'(\tilde{g})\, d_\psi g\, d_\psi x \\
 & = \eps^{\an}(\tfrac12, \pi, \psi) (\tilde W_1,\tilde W_2')_{\tilde\pi}[\tilde W_2,\tilde W_1']_{\tilde\pi}.
\end{align*}
Now Theorem \ref{thm: Mpn} follows from \eqref{eq: prop_Mp}.

To justify the manipulation, we show the convergence of
\[
 \int_{N'\bs G'}\int_{N'\bs G'}\int_{N'}\abs{\tilde W_1(\tilde{g})\whitform(W_1',\Phi,\tilde x,\tfrac12)
 [\tilde\pi(\tilde{g}^{-1}u\tilde{x})\tilde W_2,\tilde W_2']_{\tilde\pi}} \,du\, dg\, dx.
\]
Let $T'$ be the maximal torus of $G'$ consisting of diagonal matrices, $B'$ the Borel subgroup of $G'$ consisting of upper triangular matrices, and $K' = G' \cap \GL_{2n}(\mathcal{O})$ the standard maximal compact subgroup of $G'$.
As in the proof of Theorem \ref{thm: GLn}, using the Iwasawa decomposition we can rewrite the above integral as
\begin{multline*}
\int_{K'}\int_{K'}\int_{T'}\int_{T'}\int_{N'}\abs{\tilde W_1(\tilde{t}_1\tilde{k}_1)
\whitform(W_1',\Phi,\tilde t_2\tilde k_2,\tfrac12)[\tilde\pi(\tilde{k}_1^{-1}\tilde t_1^{-1}u\tilde t_2\tilde k_2)\tilde W_2,\tilde W_2']_{\tilde\pi}}
\\ \times \modulus_{B'}(t_1t_2)^{-1}\, du\, dt_1\, dt_2 \, dk_1\, dk_2.
\end{multline*}
Since $\tilde\pi$ is square-integrable, by \cite[Corollaire III.1.2]{MR1989693} (or rather its analogue for the metaplectic group),
for any $d>0$, the above integral is bounded by a constant multiple of
\begin{multline*}
\int_{K'}\int_{K'}\int_{T'}\int_{T'}\int_{N'}\abs{\tilde W_1(\tilde{t}_1\tilde{k}_1)
\whitform(W_1',\Phi,\tilde t_2\tilde k_2,\tfrac12)}\Xi^{G'}(t_1^{-1}ut_2)\sigma'(t_1^{-1}ut_2)^{-d}
\\ \times \modulus_{B'}(t_1t_2)^{-1}\,du\, dt_1\, dt_2 \, dk_1\,dk_2
\end{multline*}
where $\Xi^{G'}$ is the Harish-Chandra standard spherical function on $G'$ and $\sigma'(g)=\max(1,\log\abs{g_{i,j}})$ for $g\in G'$.
By \cite[Proposition II.4.5]{MR1989693} for any $d'>0$ we can choose $d>0$ such that the above is bounded by a constant multiple of
\[
\int_{K'}\int_{T'}\abs{\whitform(W_1',\Phi,\tilde t_2\tilde k_2,\tfrac12)}\modulus_{B'}(t_2)^{-\frac12}
\int_{K'}\int_{T'}\abs{\tilde W_1(\tilde{t}_1\tilde{k}_1)}\sigma'(t_1^{-1}t_2)^{-d'}
\modulus_{B'}(t_1)^{-\frac12}\, dt_1\, dk_1\, dt_2 \,dk_2.
\]
By the Cauchy--Schwarz inequality, the integral over $t_1, k_1$ is bounded by the square-root of
\[
\int_{K'}\int_{T'}\abs{\tilde W_1(\tilde{t}_1\tilde{k}_1)}^2\modulus_{B'}(t_1)^{-1}\, dt_1\, dk_1\times
\int_{K'}\int_{T'}\sigma'(t_1^{-1}t_2)^{-2d'}\, dt_1\, dk_1=
C_{d'}\int_{N'\bs G'}\abs{\tilde W_1(\tilde{g})}^2\, dg'
\]
where $C_{d'}<\infty$ provided that $2d'>n$. It remains to show the absolute convergence of
\[
\int_{K'}\int_{T'}\whitform(W_1',\Phi,\tilde{t}_2\tilde{k}_2,\tfrac12)\modulus_{B'}(t_2)^{-\frac12}\, dt_2\, dk_2.
\]
This follows (for any unitarizable $\pi$) from \cite[Lemma 4.6]{1401.0198}.
\end{proof}

\section{Odd special orthogonal groups --- the generic case}
\label{s: so}

Let $\OO(2n+1)$ be the (split) orthogonal group of the quadratic space $(F^{2n+1}, b(x,y)={}^t x w_{2n+1} y)$:
\[
 \OO(2n+1)=\{g\in\GL_{2n+1}: {}^tgw_{2n+1}g=w_{2n+1}\}.
\]
Let $\SO(2n+1)=\OO(2n+1)\cap\SL(2n+1)$ be the special orthogonal group.
Let $\ON$ be the maximal unipotent subgroup of $\OG$ consisting of upper unitriangular matrices.
Let $\OP=\OM\ltimes\OU$ be the Siegel parabolic subgroup of $\OG$ with Levi component $\OM=\olevi(\GL_n)$ where $\olevi(m)=\diag(m,1,m^*)$
for $m \in \GL_n$, and unipotent radical $\OU$, so that $\ON=\olevi(N_{\GL_n})\ltimes \OU$.
Here $N_{\GL_n}$ is the group of upper unitriangular matrices in $\GL_n$.

By the results of Jiang--Soudry \cite{MR1983781, MR2058617}\footnote{We emphasize that this is proved independently of the local
Langlands correspondence for $\GL_m$.}, there exists a unique injection
\[
 \JSlift : \Irr_{\gen}\SO(2n+1) \longrightarrow \Irr\GL_{2n}
\]
characterized by the equalities
\begin{equation}
\label{eq: gamma}
 \gamma^{\an}(s, \JSlift(\sigma) \times \tau, \psi) = \gamma^{\an}(s, \sigma \times \tau, \psi)
\end{equation}
for any $\sigma \in \Irr_{\gen}\SO(2n+1)$ and $\tau \in \Irr_{\gen}\GL_m$, $m\ge1$.
Here the $\gamma$-factor on the right is the one defined in \cite{Kaplan}, based on \cite{MR1169228}.
Moreover, this injection restricts to a bijection
\[
 \JSlift : \Irr_{\sqr,\gen}\SO(2n+1) \longrightarrow \Irr_{\msqr}\GL_{2n}.
\]

Let $\Omega_{\psi}$ be the Weil representation of $\Mp_n \times \OO(2n+1)$ with respect to $\psi$.
(We follow the conventions of \cite{MR2059223,MR3047069}.)
Up to equivalence, it depends only on the $(F^*)^2$-orbit of $\psi$.
For any $\sigma' \in \Irr \OO(2n+1)$, the maximal $\sigma'$-isotypic quotient of $\Omega_\psi$ is of the form
\[
 \Thetaone_\psi(\sigma') \otimes \sigma'
\]
for some smooth representation $\Thetaone_\psi(\sigma')$ of $\Mp_n$ of finite length
(see \cite[p.~46, Lemme and p.~69, Th\'eor\`eme principal]{MR1041060}).
Let $\thetaone_\psi(\sigma')$ be the maximal semisimple quotient of $\Thetaone_\psi(\sigma')$.
The Howe duality conjecture, which was proved by Waldspurger \cite{MR1159105} for $p \ne 2$,
asserts that either $\Thetaone_\psi(\sigma')$ is zero or $\Thetaone_\psi(\sigma')$ admits a unique irreducible quotient, i.e.,
$\thetaone_\psi(\sigma')$ is irreducible.\footnote{The proof of the Howe duality conjecture in general was recently completed by Gan--Takeda \cite{1407.1995} for any $p$.
This streamlines the proof of Proposition \ref{prop: thetatemp}.
Nevertheless, we keep the original argument since it will also be used in the proof of Proposition \ref{prop: theta-real}.}
Let $\sigma \in \Irr \SO(2n+1)$.
Then by \cite[Corollary 6.4]{MR2999299}, there exists a unique extension $\sigma'$ of $\sigma$ to $\OO(2n+1)$ such that
$\Thetaone_\psi(\sigma')$ is non-zero. Set
\[
 \Thetaone_\psi(\sigma) = \Thetaone_\psi(\sigma'), \qquad
 \thetaone_\psi(\sigma) = \thetaone_\psi(\sigma').
\]
Similarly, for any $\tilde \pi \in \Irr \Mp_n$,
we have a smooth representation $\Thetatwo_\psi(\tilde \pi)$ of $\OO(2n+1)$ of finite length
and its maximal semisimple quotient $\thetatwo_\psi(\tilde \pi)$.
We regard $\Thetatwo_\psi(\tilde \pi)$ and $\thetatwo_\psi(\tilde \pi)$ as representations of $\SO(2n+1)$ by restriction.
Once again for $p\ne2$, $\thetatwo_\psi(\tilde \pi)$ is either zero or irreducible.

\begin{proposition}
\label{prop: theta}
\begin{enumerate}

\item
If $\sigma \in \Irr_{\gen}\SO(2n+1)$, then $\thetaone_\psi(\sigma)\in\Irr_{\genpsi{\psi_{\tilde N}^{-1}}}\Mp_n$.

\item Let $\tilde \pi \in \Irr_{\genpsi{\psi_{\tilde N}^{-1}}}\Mp_n$.
Then precisely one irreducible subquotient of $\Thetatwo_\psi(\tilde \pi)$ is generic.
In particular, $\Thetatwo_\psi(\tilde \pi)$ is non-zero.
\end{enumerate}
\end{proposition}

\begin{proof}
Let $\sigma \in \Irr_{\gen}\SO(2n+1)$.
For any irreducible constituent $\tilde \pi$ of $\thetaone_\psi(\sigma)$,
we have a non-zero $\Mp_n \times \SO(2n+1)$-equivariant map
\[
 \Omega_{\psi} \longrightarrow \tilde \pi \otimes \sigma \longrightarrow \tilde \pi \otimes \Ind^{\SO(2n+1)}_{\ON}(\psi_{\ON}),
\]
where $\psi_{\ON}$ is a non-degenerate character of $\ON$.
By \cite[Proposition 2.1]{MR2059223}, the twisted Jacquet module $(\Omega_\psi)_{\ON, \psi_{\ON}}$ of $\Omega_\psi$ is isomorphic as an $\Mp_n$-module to
\[
 \ind^{\Mp_n}_{\tilde N}(\psi_{\tilde N}^{-1}),
\]
where $\ind$ denotes compact induction.
Hence
\begin{equation} \label{eq: thetawhit}
\begin{aligned}
 \Hom_{\Mp_n \times \SO(2n+1)}(\Omega_{\psi}, \tilde \pi \otimes \Ind^{\SO(2n+1)}_{\ON}(\psi_{\ON}))
 & \simeq \Hom_{\Mp_n \times \ON}(\Omega_{\psi}, \tilde \pi \otimes \psi_{\ON}) \\
 & \simeq \Hom_{\Mp_n}((\Omega_{\psi})_{\ON, \psi_{\ON}}, \tilde \pi) \\
 & \simeq \Hom_{\Mp_n}(\ind^{\Mp_n}_{\tilde N}(\psi_{\tilde N}^{-1}), \tilde \pi) \\
 & \simeq \Hom_{\Mp_n}(\tilde \pi^\vee, \Ind^{\Mp_n}_{\tilde N}(\psi_{\tilde N})),
\end{aligned}
\end{equation}
so that $\tilde \pi^\vee$ is $\psi_{\tilde N}$-generic.
Set $\delta = \diag(I_n, -I_n)$.
The adjoint action $\Ad_\delta$ of $\delta$ on $\Sp_n$ lifts uniquely to an action $\widetilde{\Ad}_\delta$ on $\Mp_n$.
By \cite[p.~92, Th\'eor\`eme]{MR1041060} combined with \cite[Corollaire 4.3.3]{MR3053009}, we have $\tilde \pi^\vee = \tilde \pi \circ \widetilde{\Ad}_\delta$.
Since $\psi_{\tilde N}$ and $\psi_{\tilde N}^{-1} \circ \widetilde{\Ad}_\delta$ are conjugate,
$\tilde \pi$ is $\psi_{\tilde N}^{-1}$-generic.
On the other hand, by \cite{MR818351}, the supercuspidal support of $\tilde \pi$ is uniquely determined by $\sigma$.
Since any element in $\Irr_{\genpsi{\psi_{\tilde N}^{-1}}}\Mp_n$ is uniquely determined by its supercuspidal support (cf.~\cite{MR2366363}),
$\thetaone_\psi(\sigma)$ must be isotypic. But by \cite{MR2782253}, $\thetaone_\psi(\sigma)$ is multiplicity-free and hence irreducible.
This proves the first part.

Let $\tilde \pi \in \Irr_{\genpsi{\psi^{-1}_{\tilde N}}}\Mp_n$.
Then
\begin{align*}
 \Hom_{\C}(\Thetatwo_{\psi}(\tilde \pi)_{\ON, \psi_{\ON}}, \C)
 & \simeq \Hom_{\ON}(\Thetatwo_{\psi}(\tilde \pi), \psi_{\ON}) \\
 & \simeq \Hom_{\SO(2n+1)}(\Thetatwo_{\psi}(\tilde \pi), \Ind^{\SO(2n+1)}_{\ON}(\psi_{\ON})) \\
 & \simeq \Hom_{\Mp_n \times \SO(2n+1)}(\Omega_{\psi}, \tilde \pi \otimes \Ind^{\SO(2n+1)}_{\ON}(\psi_{\ON}))\\
 & \simeq \Hom_{\Mp_n}(\tilde \pi^\vee, \Ind^{\Mp_n}_{\tilde N}(\psi_{\tilde N}))
\end{align*}
by \eqref{eq: thetawhit}. The second part follows.
\end{proof}

\begin{proposition}
\label{prop: thetatemp}
Let $* = \temp$ or $\sqr$.
\begin{enumerate}

\item
If $\sigma \in \Irr_{*, \gen}\SO(2n+1)$, then $\Thetaone_\psi(\sigma)\in\Irr_{*, \genpsi{\psi_{\tilde N}^{-1}}}\Mp_n$.

\item
If $\tilde \pi \in \Irr_{*,\genpsi{\psi_{\tilde N}^{-1}}}\Mp_n$, then $\Thetatwo_\psi(\tilde \pi)\in\Irr_{*, \gen}\SO(2n+1)$.
\end{enumerate}
\end{proposition}

\begin{proof}
By the result of Gan--Savin \cite[Theorem 8.1]{MR2999299}, for any $\sigma \in \Irr_{\temp} \SO(2n+1)$ (resp.~$\sigma \in \Irr_{\sqr} \SO(2n+1)$),
$\Thetaone_\psi(\sigma)$ is semisimple (i.e., $\Thetaone_\psi(\sigma) = \thetaone_\psi(\sigma)$) and tempered (resp.~square-integrable).
Hence the first part follows from Proposition \ref{prop: theta}.

By [loc.~cit.], the same holds for $\Thetatwo_\psi(\tilde \pi)$.
Thus, to prove the second part, it remains to show that for $\tilde \pi \in \Irr_{\temp,\genpsi{\psi_{\tilde N}^{-1}}}\Mp_n$,
any irreducible constituent $\sigma$ of $\thetatwo_\psi(\tilde \pi)$ is generic.
We have a non-zero $\Mp_n \times \SO(2n+1)$-equivariant map
\[
 \Omega_{\psi} \longrightarrow \tilde \pi \otimes \sigma \longrightarrow \Ind^{\Mp_n}_{\tilde N}(\psi_{\tilde N}^{-1})  \otimes \sigma.
\]
By \cite[Proposition 2.1]{MR2059223}, the twisted Jacquet module $(\Omega_\psi)_{\tilde N, \psi_{\tilde N}^{-1}}$ of $\Omega_\psi$
is isomorphic as an $\SO(2n+1)$-module to
\[
 \ind^{\SO(2n+1)}_\bes(\psi_\bes),
\]
where $\bes$ is the Bessel subgroup of $\SO(2n+1)$ given by
\[
 \bes=\left\{ \left( \begin{smallmatrix} u&*&*\\ &\begin{smallmatrix} t&&\\&1&\\&&t^{-1}\end{smallmatrix}&*\\ &&u^*\end{smallmatrix} \right) :
  u \in N_{\GL_{n-1}}, \, t \in F^* \right\}
\]
and $\psi_\bes$ is the character of $\bes$ given by $\psi_\bes(b) = \psi(b_{1,2} + \dots + b_{n-2,n-1} + b_{n-1,n+1})$.
As in \eqref{eq: thetawhit}, we compute
\[
 \Hom_{\Mp_n \times \SO(2n+1)}(\Omega_{\psi}, \Ind^{\Mp_n}_{\tilde N}(\psi_{\tilde N}^{-1})  \otimes \sigma)
 \simeq \Hom_{\SO(2n+1)}(\sigma^\vee, \Ind^{\SO(2n+1)}_\bes(\psi_\bes)),
\]
so that $\sigma^\vee$ has a Bessel model, i.e., the Hom space on the right is non-zero.
Since $\sigma^\vee = \sigma$ by \cite[p.~91, Th\'eor\`eme]{MR1041060} and we already know that $\sigma$ is tempered,
the following proposition concludes the proof.
\end{proof}

\begin{proposition}\label{prop: supp}
Let $\sigma \in \Irr_{\temp}\SO(2n+1)$.
If $\sigma$ has a Bessel model, then $\sigma$ is generic.
\end{proposition}

We remark that the proposition does not hold for an arbitrary $\sigma$; evidently the trivial representation of $\SO(3)$
has a Bessel model, but is not generic.

To prove Proposition \ref{prop: supp}, we need to introduce more notation.
Let $T_1=\{\inj_{1,1}(t) := \diag(t,I_{2n-1},t^{-1}) : t \in F^* \}$.
Define the one-parameter root subgroup $\ON_{i,j}=\{\inj_{i,j}(x): x\in F\}$ of $\ON$ so that
$t\inj_{i,j}(x)t^{-1}=\inj_{i,j}(\frac{t_i}{t_j}x)$ for $t=\diag(t_1,\ldots,t_{2n+1})\in \OG$.
We recall the definition of the Bessel subgroup in \cite{MR3047069}:
$$\bes'=(T_1\ltimes(\olevi(\Ncirc)\ltimes\prod_{i=2}^n\ON_{i,1}))\ltimes U^\sharp,$$
where $\Ncirc=\{\sm{1}{}{}{u}: u\in N_{\GL_{n-1}}\}\subset N_{\GL_n}$ and $U^\sharp=\{u\in \OU: u_{1,n+1}=0\}$.
For instance, for $n=3$, $\bes'$ consists of the matrices in $\SO(7)$ of the form
\[
\left(\begin{smallmatrix}*&0&0&0&*&*&*\\ *&1&*&*&*&*&*\\ *&0&1&*&*&*&*\\
0&0&0&1&*&*&0\\0&0&0&0&1&*&0\\0&0&0&0&0&1&0\\0&0&0&0&*&*&*\end{smallmatrix}\right).
\]
This group is conjugate by a Weyl element in $\OG$ to the `standard' Bessel subgroup $\bes$ in the proof of Proposition \ref{prop: thetatemp}.
Define a function $\Psi$ on $\SO(2n+1)$ by
\[
\Psi(g)=\psi(g_{2,3}+\dots+g_{n,n+1}).
\]
Then $\psi_{\bes'} := \Psi\rest_{\bes'}$ is a character of $\bes'$.
We also introduce some auxiliary groups.
For $k=2,\ldots,n$, let $\bes^k$ be the group generated by $T_1$, $\olevi(\Ncirc)$, $\ON_{i,1}$ ($i<k$), $\ON_{1,j}$ ($j>k$) and $\OU$,
with $\psi_{\bes^k}$ being the restriction of $\Psi$ to $\bes^k$.
It is easy to check that $\psi_{\bes^k}$ is a character of $\bes^k$.
For $b\in \bes^k$, define $\wgt(b):=\abs{b_{1,1}}$; then $\wgt$ is a character of $\bes^k$.

For a closed subgroup $H$ of $\OG$ and a character $\chi$ of $H$, let $C^{\smth}(H\bs \OG,\chi)$ denote the space of left
$(H,\chi)$-equivariant functions on $\OG$ which are smooth under the right translation of $\OG$.
For $\sigma \in \Irr \OG$, we say that $\sigma$ has a $(H,\chi)$ model if there is a non-zero $\OG$-equivariant map
from the space of $\sigma$ to $C^{\smth}(H\bs \OG,\chi)$.
Note that $\sigma$ has a Bessel model if and only if it has a $(\bes',\psi_{\bes'})$ model.

The following lemma is a special case of \cite[Lemma~2.2]{MR1671452} (whose argument goes back to \cite{MR519356}).
For completeness we give a direct proof in the case at hand.
\begin{lemma} \label{L: supp1}
Let $\sigma \in \Irr \OG$.
Then $\sigma$ has a $(\bes',\psi_{\bes'})$ model if and only if it has a $(\bes^n,\wgt\psi_{\bes^n})$ model.
When $k=2,\ldots,n-1$ and $s \in \C$, $\sigma$ has a $(\bes^k,\wgt^s\psi_{\bes^k})$ model if and only if it has a
$(\bes^{k+1},\wgt^{s-1}\psi_{\bes^{k+1}})$ model.
\end{lemma}

\begin{proof}
It is convenient to use the notation $\phi*_{i,j}f=\int_F f(\,\cdot \,\inj_{i,j}(x))\phi(x)\,dx$ for a continuous function
$f$ on $\OG$ and a Schwartz function $\phi$ on $F$. We treat the first statement of the claim.

Let $f\in C^{\smth}(\bes'\bs \OG, \psi_{\bes'})$ and $\phi \in C_c^\infty(F)$. Observe that
 \begin{multline*}
 \phi*_{n,1}f(\inj_{1,n+1}(x))=\int_F f(\inj_{1,n+1}(x)\inj_{n,1}(y))\phi(y)\,dy
 \\
 =\int_F f([\inj_{1,n+1}(x),\inj_{n,1}(y)]\inj_{n,1}(y)\inj_{1,n+1}(x))\phi(y)\,dy.
 \end{multline*}
Since the commutator $[\inj_{1,n+1}(x),\inj_{n,1}(y)]$ belongs to $\bes'$ and $\psi_{\bes'}([\inj_{1,n+1}(x),\inj_{n,1}(y)])=\psi(-xy)$
(if we choose $\inj_{i,j}$ in a compatible way), the above integral is
$$\int_F f(\inj_{1,n+1}(x))\phi(y)\psi(-xy)\,dy= f(\inj_{1,n+1}(x))\hat\phi(x),$$
where $\hat\phi \in C_c^\infty(F)$ is the Fourier transform of $\phi$ with respect to $\psi^{-1}$.
Thus $\phi*_{n,1}f(\inj_{1,n+1}(x))$ is a Schwartz function in $x$. This implies that $f(\inj_{1,n+1}(x))$ is a Schwartz function in $x$.
We can define a linear form
$$T_{1,n+1}(f)=\int_F f(\inj_{1,n+1}(x))\,dx.$$
It is clear that $T_{1,n+1}(f(\,\cdot\,g)) \in C^{\smth}(\bes^n\bs \OG, \wgt\psi_{\bes^n})$ (as a function of $g$).
Moreover the above calculation shows that
\begin{equation}\label{eq: supp1}
T_{1,n+1}(\phi*_{n,1}f)=\int_F \phi*_{n,1}f(\inj_{1,n+1}(x))\,dx=\hat\phi*_{1,n+1}f(e).
\end{equation}
If $\sigma$ has a $(\bes',\psi_{\bes'})$ model, then the image of the model under $T_{1,n+1}$ is non-trivial by \eqref{eq: supp1}
and hence $\sigma$ has a $(\bes^n,\wgt\psi_{\bes^n})$ model.
The reverse direction is similar, as we have for $f\in C^{\smth}(\bes^n\bs \OG, \wgt\psi_{\bes^n})$ and $\phi\in C_c^\infty(F)$:
$$T_{n,1}(\phi*_{1,n+1}f):=\int_F \phi*_{1,n+1}f(\inj_{n,1}(x))\,dx=\hat\phi*_{n,1}f(e).$$
The proof of the second statement is similar, using the fact that when $k=2,\ldots,n-1$, for $f\in C^{\smth}(\bes^k\bs \OG, \wgt^s\psi_{\bes^k})$
and $\phi\in C_c^{\infty}(F)$:
$$T_{k,1}(\phi*_{1,k+1}f):=\int_F \phi*_{1,k+1}f(\inj_{k,1}(x))\,dx=\hat\phi*_{k,1}f(e),$$
and for $f\in C^{\smth}(\bes^{k+1}\bs \OG, \wgt^{s-1}\psi_{\bes^{k+1}})$ and $\phi\in C_c^{\infty}(F)$:
$$T_{1,k+1}(\phi*_{1,k+1}f):=\int_F \phi*_{k,1}f(\inj_{1,k+1}(x))\,dx=\hat\phi*_{1,k+1}f(e).$$
The lemma follows.
\end{proof}

\begin{proof}[Proof of Proposition \ref{prop: supp}]
If $\sigma$ has a Bessel model, then it also has a $(\bes^2, \wgt^{n-1}\psi_{\bes^2})$ model by Lemma \ref{L: supp1}.
Note that $\bes^2=T_1\ltimes N^\sharp$ where $N^\sharp=\{u\in \ON: u_{1,2}=0\}$.
For a character $\psi$ of a subgroup $J$ of $\SO(2n+1)$, denote by $\sigma_{(J,\psi)}$ the co-invariants of $\sigma$ with respect to $(J,\psi)$.
Write $\sigma_{(\bes^2, \wgt^{n-1}\psi_{\bes^2})}=A_{(T_1, \wgt^{n-1})}$ where $A=\sigma_{(N^\sharp,\psi_{\bes^2})}$.
Consider $A$ as a $T_1\ltimes \ON$-module.
If $\ON_{1,2}$ does not act trivially on $A$, then there is a non-trivial character $\psi$ of $\ON_{1,2}$ for which $A_{(\ON_{1,2},\psi)}\ne0$.
Hence $\sigma$ is generic.
Otherwise, $A_{(T_1,\wgt^{n-1})}$ factors through the Jacquet module of $\sigma$ with respect to the standard parabolic subgroup with Levi component $\GL_1\times\SO(2n-1)$.
Hence the character $\inj_{1,1}(t) \mapsto \abs{t}^{-\frac12}$ occurs in the above Jacquet module, in contradiction with the temperedness assumption on $\sigma$.
\end{proof}

It follows from Propositions \ref{prop: theta} and \ref{prop: thetatemp} that the Howe duality defines an injection
\[
 \thetaone_\psi : \Irr_{\gen}\SO(2n+1) \longrightarrow \Irr_{\genpsi{\psi_{\tilde N}^{-1}}}\Mp_n
\]
which restricts to a bijection
\[
 \Thetaone_\psi : \Irr_{*, \gen} \SO(2n+1) \longrightarrow \Irr_{*, \genpsi{\psi_{\tilde N}^{-1}}}\Mp_n
\]
for $* = \temp$ or $\sqr$.
The inverse of $\Thetaone_\psi$ is
\[
 \Thetatwo_\psi : \Irr_{*, \genpsi{\psi_{\tilde N}^{-1}}}\Mp_n \longrightarrow \Irr_{*, \gen} \SO(2n+1).
\]

\begin{proposition}
\label{prop: gamma2}
We have
\[
 \gamma(s, \thetaone_\psi(\sigma) \times \tau, \psi) = \gamma(s, \sigma \times \tau, \psi)
\]
for $\sigma \in \Irr_{\gen}\SO(2n+1)$ and $\tau \in \Irr_{\gen}\GL_m$, $m\ge1$.
\end{proposition}

\begin{proof}
This was stated in \cite[Corollary 11.3]{MR2999299} but the deduction from [ibid., Proposition 11.1] is incorrect since it assumes a theory
of $\gamma$-factors for irreducible but not necessarily generic representations.
The proof can be corrected as follows.

Let $\tilde \pi \in \Irr_{\genpsi{\psi_{\tilde N}^{-1}}}\Mp_n$ and $\tau \in \Irr_{\gen}\GL_m$.
Recall that by Proposition \ref{prop: theta}, $\thetatwo_\psi(\tilde \pi)$ is non-zero.
We need to show that
\[
 \gamma(s, \tilde \pi \times \tau, \psi) = \gamma(s, \sigma \times \tau, \psi)
\]
for the irreducible generic constituent $\sigma$ of $\thetatwo_\psi(\tilde \pi)$ if it exists.
Let $\tau_{1, \psi} \otimes \dots \otimes \tau_{k, \psi} \otimes \tilde \pi_0$ be the supercuspidal support of $\tilde \pi$, where
$\tau_i \in \Irr_{\cusp} \GL_{n_i}$ and $\tilde \pi_0 \in \Irr_{\cusp} \Mp_{n_0}$ with $n_0 = n - n_1 - \dots - n_k$.
Here we define a representation $\tau_{i,\psi}$ of the double cover of $\GL_{n_i}$ (with multiplication law determined by the Hilbert symbol)
by $\tau_{i,\psi}(g, \epsilon) = \epsilon \gamma_{\psi}(\det g) \tau_i(g)$ for $(g, \epsilon) \in \GL_{n_i} \times \{ \pm 1 \}$, where $\gamma_\psi$
is Weil's factor with respect to $\psi$.
Since $\tilde \pi$ is $\psi_{\tilde N}^{-1}$-generic, $\tilde \pi_0$ is also $\psi_{\tilde N}^{-1}$-generic where we regard $\psi_{\tilde N}$ as
a non-degenerate character of the maximal unipotent subgroup of $\Mp_{n_0}$ by restriction.
Let $\sigma_0 = \Thetatwo_\psi(\tilde \pi_0) \in \Irr_{\sqr, \gen} \SO(2n_0+1)$.
By \cite[Theorem 2.2]{MR1983781}, $\sigma_0$ is supercuspidal unless $n_0 = 1$ and $\tilde \pi_0$ is the odd Weil representation of $\Mp_1$
with respect to $\psi$, in which case $\sigma_0$ is the Steinberg representation of $\SO(3)$.
Then by Kudla's supercuspidal theorem \cite{MR818351}, the supercuspidal support of $\sigma$ is of the form
\[
\begin{cases}
 \tau_1 \otimes \dots \otimes \tau_k \otimes \sigma_0 & \text{unless $n_0=1$ and $\tilde \pi_0$ is the odd Weil representation,} \\
 \tau_1 \otimes \dots \otimes \tau_k \otimes \abs{\,\cdot\,}^{\frac12} & \text{if $n_0=1$ and $\tilde \pi_0$ is the odd Weil representation.}
\end{cases}
\]
Hence, by the multiplicativity of $\gamma$-factors \cite{Kaplan}, we may assume that both $\tilde \pi$ and $\tau$ are supercuspidal.
Then $\sigma := \Thetatwo_\psi(\tilde \pi) = \thetatwo_\psi(\tilde \pi)$ is irreducible and generic by Proposition \ref{prop: thetatemp}.

We first assume that $\sigma$ is supercuspidal, which is the case unless $n=1$ and $\tilde \pi$ is the odd Weil representation of $\Mp_1$
with respect to $\psi$ (see \cite[Theorem 2.2]{MR1983781}).
Choose a totally complex number field $k$ such that $k_{v_0} = F$ for some place $v_0$ of $k$.
By replacing $\psi$ by a character in its $(F^*)^2$-orbit if necessary, we may assume that there exists a non-trivial character $\varPsi$ of
$k \bs \A_k$ (where $\A_k$ is the adele ring of $k$) such that $\varPsi_{v_0} = \psi$.
We may also assume that $\varPsi_v$ is trivial on $\mathcal{O}_v$ for all finite places $v$ of $k$.
We claim that there exists an irreducible globally $\varPsi_{\tilde N}^{-1}$-generic cuspidal automorphic representation $\tilde\Pi$ of
$\Mp_n(\A_k)$ such that $\tilde\Pi_{v_0} = \tilde\pi$, and $\tilde\Pi_v$ is a subquotient of a principal series representation of
$\Mp_n(k_v)$ for any other place $v \ne v_0$.

The argument is standard (cf.~\cite{MR1904086, MR2369492}).
We choose $f_v\in C_c^\infty(\tilde N(k_v)\bs\tilde G(k_v),\varPsi_{\tilde N,v}^{-1})$ of the form $f_v=\int_{N'(k_v)} h_v(u\,\cdot\,)
\varPsi_{\tilde N,v}(u)\, du$, where $h_v$ is a genuine function in $C_c^\infty(\tilde G(k_v))$ as follows.
At $v_0$ we take $h_{v_0}$ to be a matrix coefficient of $\tilde\pi$ such that $f_v(e)=1$
(cf.~\cite[Lemma 4.4]{MR2369492}).
At the finite places $v\ne v_0$ of odd residual characteristic we take $h_v$ to be the function defined for $g \in \Mp_n(k_v)$ by
\[
 h_v(g) =
 \begin{cases}
  \epsilon & \text{if $g \in \epsilon \cdot \Sp_n(\mathcal{O}_v)$ with $\epsilon \in \{ \pm 1 \}$,} \\
  0 & \text{otherwise,}
 \end{cases}
\]
where $\Sp_n(\mathcal{O}_v)$ is regarded as a subgroup of $\Mp_n(k_v)$ and $\{ \pm 1 \}$ is the kernel of the projection
$\Mp_n(k_v) \rightarrow \Sp_n(k_v)$.
At the finite places $v \ne v_0$ of even residual characteristic we take $h_v$ in the part of $C_c^\infty(\tilde G(k_v))$
corresponding to the Bernstein component of a principal series representation, such that $f_v(e)=1$.
(Once again, \cite[Lemma 4.4]{MR2369492} guarantees that such an $h_v$ exists since any non-zero $h_v$ in the Bernstein component
acts non-trivially on some irreducible principal series representation.)
At the archimedean places $v$ we take $h_v$ to be supported in a small neighborhood of $e$ and such that $f_v(e)=1$.
Let $h=\otimes_v h_v$ and $f=\otimes_v f_v$ so that $f=\int_{N'(\A_k)}h(u \, \cdot \,)\varPsi_{\tilde N}(u)\, du$.
Let $K(g)=\sum_{\gamma\in N'(k)\bs G'(k)}f(\gamma g)$.
Then as in the proof of \cite[Theorem 4.1]{MR2369492}, $K$ is cuspidal, $K(e)=f(e)\ne0$, and any irreducible representation occurring in its spectral
decomposition satisfies the required conditions.

Let $\Sigma$ be the global theta $\varPsi$-lift of $\tilde\Pi$ to $\SO(2n+1,\A_k)$.
By \cite[Proposition 3]{MR1353315}, $\Sigma$ is non-zero and globally generic.
Since we have assumed that $\sigma$ is supercuspidal, $\Sigma$ is cuspidal.
By \cite[Theorem 1.3]{MR2330445}, $\Sigma$ is irreducible.\footnote{The irreducibility of $\Sigma$ also follows from the Howe duality (cf.~\cite[Corollary 7.1.3]{MR1289491}).}
Similarly, we can find an irreducible cuspidal automorphic representation $\Tau$ of $\GL_m(\A_k)$ such that $\Tau_{v_0} = \tau$,
and $\Tau_v$ is a subquotient of a principal series representation of $\GL_m(k_v)$ for any other place $v \ne v_0$.
Now, by Kudla's results on the theta correspondence \cite{MR818351} and once again, the multiplicativity of $\gamma$-factors \cite{Kaplan}, we have
\begin{equation}
\label{eq: gamma3}
 \gamma(s, \tilde\Pi_v \times \Tau_v, \varPsi_v) = \gamma(s, \Sigma_v \times \Tau_v, \varPsi_v)
\end{equation}
for any finite place $v \ne v_0$.
The same holds for any complex place $v$ by \cite{MR1346217}.
Hence we conclude from the global functional equation that \eqref{eq: gamma3} holds for $\tilde\Pi_{v_0} = \tilde\pi$ and $\Tau_{v_0} = \tau$ as well.

If $\sigma$ is not supercuspidal, then in the above globalization,
we choose an auxiliary finite place $v_1 \ne v_0$ and $\sigma_1 \in \Irr_{\cusp, \gen} \SO(2n+1,k_{v_1})$.
As before, there exists an irreducible globally $\varPsi_{\tilde N}^{-1}$-generic cuspidal automorphic representation $\tilde\Pi$ of $\Mp_n(\A_k)$
such that $\tilde\Pi_{v_0} = \tilde\pi$, $\tilde\Pi_{v_1} = \Thetaone_{\varPsi_{v_1}}(\sigma_1)$, and $\tilde\Pi_v$ is a subquotient
of a principal series representation of $\Mp_n(k_v)$ for any other place $v \ne v_0, v_1$.
Then the rest of the argument is the same, since we have already shown \eqref{eq: gamma3} for $\tilde\Pi_{v_1}$.
\end{proof}

\begin{proposition}
\label{prop: triangle}
For any $\sigma\in\Irr_{\sqr,\gen} \SO(2n+1)$ we have $\Thetaone_\psi(\sigma) = \des_{\psi^{-1}}(\JSlift(\sigma))$.
\end{proposition}

\begin{proof}
For $\sigma\in\Irr_{\sqr,\gen}\SO(2n+1)$, set $\tilde\pi=\Thetaone_\psi(\sigma)$ and $\pi=\JSlift(\sigma)$.
Let $k$ be a number field such that $k_{v_0}=F$ for some place $v_0$ of $k$.
By replacing $\psi$ by a character in its $(F^*)^2$-orbit if necessary, we may assume that there exists a non-trivial character $\varPsi$ of
$k\bs\A_k$ such that $\varPsi_{v_0}=\psi$.
By Corollary \ref{lem: globalize2}, there exists an irreducible globally $\varPsi_{\tilde N}^{-1}$-generic cuspidal
automorphic representation $\tilde\Pi$ of $\Mp_n(\A_k)$ such that $\tilde\Pi_{v_0}=\tilde\pi$ and the global theta $\varPsi$-lift $\Sigma$ of $\tilde\Pi$ to
$\SO(2n+1,\A_k)$ is non-zero, irreducible, globally generic and cuspidal.
By \cite[Theorem 1.1]{MR2330445} the global theta $\varPsi$-lift of $\Sigma$ to $\Mp_n(\A_k)$ is equal to $\tilde\Pi$ and in particular it is non-zero.
Therefore by \cite[Theorem 1.4]{MR2330445} $L^S(\frac12,\Sigma)\ne0$ for a sufficiently large finite set $S$ of places of $k$.\footnote{Alternatively,
this also follows from the Rallis inner product formula, basic properties of the local standard $L$-factors \cite{Yam}
and the description of the generic unitary dual \cite{MR2046512}.}
Let $\Pi$ be the weak lift of $\Sigma$ to $\GL_{2n}(\A_k)$.
Then $L^S(\frac12,\Pi)=L^S(\frac12,\Sigma)\ne0$ and $\Pi_{v_0}=\pi$ by \cite[Theorem E]{MR2058617}.
Moreover, the global descent $\des_{\varPsi^{-1}}(\Pi)$ of $\Pi$ is irreducible by \cite{MR3009748}\footnote{The irreducibility of
$\des_{\varPsi^{-1}}(\Pi)$ can also be proved as follows.
By \cite[Theorem 1.7]{MR1954940}, $\des_{\varPsi^{-1}}(\Pi)$ is a multiplicity-free direct sum of irreducible globally
$\varPsi^{-1}_{\tilde N}$-generic cuspidal automorphic representations of $\Mp_n(\A_k)$.
Take any irreducible constituent $\tilde \Pi$ of $\des_{\varPsi^{-1}}(\Pi)$ and consider its global theta $\varPsi$-lift $\Sigma$ to $\SO(2n+1, \A_k)$.
By the local unramified theta correspondence, we see that the global theta $\varPsi$-lift of $\tilde \Pi$ to $\SO(2n-1, \A_k)$ vanishes, so that $\Sigma$ is cuspidal.
As in the proof of Proposition \ref{prop: gamma2}, $\Sigma$ is non-zero, irreducible and globally generic.
Then $\Pi$ is a weak lift of $\Sigma$.
By \cite[Theorem E]{MR2058617}, $\Pi_v = \JSlift(\Sigma_v)$ for all places $v$.
Since $\JSlift$ is injective, the equivalence class of $\Sigma$ is uniquely determined by $\Pi$.
So is $\tilde \Pi$ by the Howe duality.
Thus $\des_{\varPsi^{-1}}(\Pi)$ must be irreducible.}
and equivalent to $\tilde\Pi$ by \cite[Theorem 11.2]{MR2848523}.
It follows from \cite[Theorem 6.2]{1401.0198} that $\des_{\psi^{-1}}(\Pi_{v_0})$ is irreducible and equals to
$\des_{\varPsi^{-1}}(\Pi)_{v_0}$. The proposition follows.
\end{proof}

We now prove Theorem \ref{thm: bijection}.
As a consequence of Proposition \ref{prop: triangle}, we obtain a commutative diagram
\[
 \xymatrix{
 & \Irr_{\msqr} \GL_{2n} \ar@{<-}[ddl]_{\JSlift} \ar@{->}[ddr]^{\des_{\psi^{-1}}} & \\
 & & \\
 \Irr_{\sqr,\gen} \SO(2n+1) \ar@<2pt>[rr]^{\Thetaone_\psi} & & \Irr_{\sqr, \genpsi{\psi_{\tilde N}^{-1}}}\Mp_n \ar@<2pt>[ll]^{\Thetatwo_\psi}
 }
\]
where in the diagram, $\JSlift$ and $\Thetaone_\psi$ are bijective and $\Thetatwo_\psi$ is the inverse of $\Thetaone_\psi$.
Hence $\des_{\psi^{-1}}$ in the diagram is also bijective.
For $\pi \in \Irr_{\msqr} \GL_{2n}$, let $\tilde \pi = \des_{\psi^{-1}}(\pi)$ and $\sigma = \Thetatwo_\psi(\tilde \pi)$, so that $\pi = \JSlift(\sigma)$
and $\tilde \pi = \Thetaone_\psi(\sigma)$.
Then by \eqref{eq: gamma} and Proposition \ref{prop: gamma2}, we have
\[
 \gamma(s, \pi \times \tau, \psi) = \gamma(s, \sigma \times \tau, \psi) = \gamma(s, \tilde \pi \times \tau, \psi)
\]
for any $\tau \in \Irr_{\gen}\GL_m$, $m\ge1$.
This completes the proof of Theorem \ref{thm: bijection}.

Finally, we prove the formal degree conjecture for $\SO(2n+1)$ in the generic case.
We write $d_\psi^{\SO(2n+1)}$ for the Haar measure on $\SO(2n+1)$ defined in \cite{MR1458303,MR2425185}.
(We caution that unlike the case of symplectic groups, this does not coincide with our standard convention for Haar measures on subgroups of the general linear group.)

\begin{theorem}
\label{thm: so}
Let $\sigma\in\Irr_{\sqr,\gen}\SO(2n+1)$ and let $d_\sigma$ be its formal degree.
Assume that $\pi = \JSlift(\sigma)$ is of the form \eqref{eq: pi}.
Then
\[
 d_\psi^{\SO(2n+1)} = 2^k \gamma^{\an}(1, \pi, \Sym^2, \psi)d_\sigma = 2^k \gamma^{\ari}(1, \pi, \Sym^2, \psi)d_\sigma.
\]
\end{theorem}

\begin{proof}
Set $\tilde\pi:= \Thetaone_{\psi}(\sigma) = \des_{\psi^{-1}}(\pi)$.
By \cite[Theorem 15.1]{MR3166215}, we have
\[
 d_{\tilde{\pi}}/d_\psi^{\Sp_n} = c\cdot d_{\sigma}/d_\psi^{\SO(2n+1)}
\]
for some $c>0$ which does not depend on $\sigma$.
(In fact, it is proved that $c=1$ for $p \ne 2$.
For $p=2$, see the remark after [loc.~cit.].)
Thus it follows from Theorem \ref{thm: Mpn} that
\[
 d_\psi^{\SO(2n+1)} = \abs{2}^n 2^k c \gamma^{\an}(1, \pi, \Sym^2, \psi)d_\sigma = \abs{2}^n 2^k c \gamma^{\ari}(1, \pi, \Sym^2, \psi)d_\sigma.
\]
On the other hand, the theorem was already proved in \cite{MR2425185} for the Steinberg representation of $\SO(2n+1)$.
This forces $c=\abs{2}^{-n}$ and completes the proof.
\end{proof}

\section{Odd special orthogonal groups --- the general case} \label{sec: odd general}

We now prove the formal degree conjecture for $\SO(2n+1)$, under the assumption of the local Langlands correspondence,
which was established by Arthur \cite{MR3135650} conditionally on the stabilization of the twisted trace formula.
The local Langlands correspondence (in the square-integrable case) asserts that there exists a partition
\[
 \Irr_{\sqr} \SO(2n+1) = \coprod_{\phi} \Pi_\phi
\]
into $L$-packets, where the disjoint union on the right-hand side runs over equivalence classes of square-integrable $L$-parameters
$\phi : \WD_F \rightarrow \Sp_n(\C)$.
Here we say that a continuous homomorphism $\phi : \WD_F \rightarrow \Sp_n(\C)$ is an $L$-parameter if $\phi$ is semisimple and
$\phi\rest_{\SL(2,\C)}$ is algebraic, and that $\phi$ is square-integrable if the centralizer $S_\phi$ of the image of $\phi$ in $\Sp_n(\C)$ is finite.
Moreover, there exists a bijection
\[
 \Pi_\phi \longrightarrow \widehat{\mathcal{S}}_{\phi},
\]
where $\mathcal{S}_\phi = S_\phi / \{ \pm I_{2n}\}$ and $\widehat{\mathcal{S}}_{\phi}$ is the group of characters of $\mathcal{S}_\phi$.
Note that $S_\phi$ is an elementary abelian $2$-group.

The above bijection satisfies the following properties.
We write $\sprod{\cdot}{\sigma}$ for the character of $\mathcal{S}_\phi$ associated to $\sigma \in \Pi_\phi$.
For $s \in \mathcal{S}_\phi$, set
\[
 \Theta_{\phi}^s = \sum_{\sigma \in \Pi_\phi} \sprod{s}{\sigma} \Theta_{\sigma},
\]
where $\Theta_{\sigma}$ is the character of $\sigma$.
We may assume that $\phi = \xi \circ \phi'$, where $\phi' : \WD_F \rightarrow \Sp_a(\C) \times \Sp_b(\C)$ is a square-integrable $L$-parameter
and $\xi : \Sp_a(\C) \times \Sp_b(\C) \hookrightarrow \Sp_n(\C)$ is an embedding with $a+b=n$, and that $s$ is the image of
$(I_{2a}, -I_{2b}) \in \Sp_a(\C) \times \Sp_b(\C)$ in $\mathcal{S}_\phi$.
Set
\[
 \Theta^{\mathrm{st}}_{\phi'} = \sum_{\sigma' \in \Pi_{\phi'}} \Theta_{\sigma'},
\]
where $\Pi_{\phi'} \subset \Irr_{\sqr}(\SO(2a+1) \times \SO(2b+1))$ is the $L$-packet associated to $\phi'$.
Then $\Theta^{\mathrm{st}}_{\phi'}$ is a stable distribution and $\Theta_{\phi}^s$ is the transfer of $\Theta^{\mathrm{st}}_{\phi'}$, i.e.,
\begin{equation}
\label{eq: transfer}
 \Theta_\phi^s(f) = \Theta^{\mathrm{st}}_{\phi'}(f')
\end{equation}
for $f \in C_c^\infty(\SO(2n+1))$ and $f' \in C_c^\infty(\SO(2a+1) \times \SO(2b+1))$ which have matching orbital integrals
(see \cite[\S 1.4]{MR909227}, \cite[\S 5.5]{MR1687096}; see also \cite[\S 5.3]{MR1687096}, \cite[\S 2.1]{MR3135650} for the normalization of transfer factors).

\begin{corollary}
\label{cor: nongen}
If we admit the local Langlands correspondence for $\SO(2n+1)$, then the formal degree conjecture holds for $\SO(2n+1)$.
Namely, we have
\[
 d_\psi^{\SO(2n+1)} = \abs{S_\phi} \gamma^{\ari}(1, \sigma, \Ad, \psi)d_{\sigma}
\]
for any square-integrable $L$-parameter $\phi: \WD_F \rightarrow \Sp_n(\C)$ and any $\sigma \in \Pi_\phi$.
\end{corollary}

\begin{proof}
Let $\phi: \WD_F \rightarrow \Sp_n(\C)$ be a square-integrable $L$-parameter.
It follows from \eqref{eq: transfer} and \cite[Corollary 9.10]{MR1070599} that
\[
 \sum_{\sigma \in \Pi_\phi} \sprod{s}{\sigma} d_\sigma = 0
\]
if $s \in \mathcal{S}_\phi$ is non-trivial.
Thus the representations in $\Pi_\phi$ have the same formal degree.
Hence, to prove the corollary, we may assume that the character $\sprod{\cdot}{\sigma}$ is trivial.
Then by \cite[Proposition 8.3.2]{MR3135650}, $\sigma$ is generic.
In fact, as explained in the proof of [loc.~cit.],
there exist a number field $k$, a place $v_0$ of $k$,
an automorphic representation $\Pi$ of $\GL_{2n}(\A_k)$,
and an irreducible globally generic cuspidal automorphic representation $\Sigma$ of $\SO(2n+1,\A_k)$ such that
\begin{itemize}
\item $k_{v_0} = F$,
\item $\Pi_{v_0}$ corresponds to $\iota \circ \phi$,
where $\iota : \Sp_n(\C) \hookrightarrow \GL_{2n}(\C)$ is a natural embedding,
\item $\Sigma$ lifts weakly to $\Pi$,
\item $\Sigma_{v_0} \in \Pi_{\phi}$ and $\sprod{\cdot}{\Sigma_{v_0}}$ is trivial, i.e., $\Sigma_{v_0} = \sigma$.
\end{itemize}
On the other hand, by \cite[Theorem E]{MR2058617},
\[
 \Pi_{v_0} = \JSlift(\Sigma_{v_0}).
\]
Hence $\JSlift(\sigma)$ corresponds to  $\iota \circ \phi$.
This reduces the corollary to Theorem \ref{thm: so}.
\end{proof}

Now let $\SO(2n+1)^-$ be the non-split special orthogonal group in $2n+1$ variables.
We will explain how to prove the formal degree conjecture for $\SO(2n+1)^-$ if we admit the local Langlands correspondence (cf.~\cite[\S 4.2]{WaldAst3472}).

For the sake of uniform notation, we write $\SO(2n+1)^+$ for the split special orthogonal group in $2n+1$ variables.
The local Langlands correspondence (in the square-integrable case) asserts that there exist a partition
\[
 \Irr_{\sqr} \SO(2n+1)^+ \, {\textstyle\coprod} \, \Irr_{\sqr} \SO(2n+1)^- = \coprod_{\phi} \Pi_\phi
\]
and a bijection
\[
 \Pi_\phi \longrightarrow \widehat{S}_\phi
\]
satisfying the following properties,
where the  the disjoint union runs over conjugacy classes of square-integrable $L$-parameters  $\phi : \WD_F \rightarrow \Sp_n(\C)$
and $\widehat{S}_\phi$ is the group of characters of $S_\phi$.
We write $\sprod{\cdot}{\sigma}$ for the character of $S_\phi$ associated to $\sigma \in \Pi_\phi$.
Set
\[
 \Pi_\phi^\pm = \{ \sigma \in \Pi_\phi : \sprod{-I_{2n}}{\sigma} = \pm 1 \}.
\]
Then $\Pi_\phi^\pm = \Pi_\phi \cap \Irr_{\sqr} \SO(2n+1)^\pm$, and the endoscopic character relations hold (see \cite[\S 4.2 and \S 4.8]{WaldAst3472}).
For example, set
\[
 \Theta_\phi^{\mathrm{st},\pm} = \sum_{\sigma \in \Pi_\phi^\pm} \Theta_\sigma.
\]
Then $\Theta_\phi^{\mathrm{st},\pm}$ is a stable distribution and $- \Theta_\phi^{\mathrm{st}, -}$ is the transfer of $\Theta_\phi^{\mathrm{st},+}$.
Hence it follows from the proof of \cite[Corollary 9.10]{MR1070599} that
\[
 \sum_{\sigma \in \Pi_\phi^+} \frac{d_\sigma}{d_\psi^+} = c \cdot \sum_{\sigma \in \Pi_\phi^-} \frac{d_\sigma}{d_\psi^-}
\]
for some $c>0$ which does not depend on $\phi$,
where $d_\psi^\pm = d_\psi^{\SO(2n+1)^\pm}$ is the Haar measure on $\SO(2n+1)^\pm$ defined in \cite{MR1458303,MR2425185}.
If $\phi$ is the $L$-parameter such that $\phi\rest_{W_F}$ is trivial and $\phi\rest_{\SL(2,\C)}$ corresponds to the regular
unipotent orbit in $\Sp_n(\C)$, then $\Pi_\phi^\pm$ is a singleton consisting of the Steinberg representation of $\SO(2n+1)^\pm$.
Let $d_0^\pm$ be the formal degree of the Steinberg representation of $\SO(2n+1)^\pm$.
Since $d_0^+/d_\psi^+ = d_0^-/d_\psi^-$ by \cite{MR2425185}, we must have $c=1$.
Also, as in the proof of Corollary \ref{cor: nongen}, the other endoscopic character relations imply that the representations
in $\Pi_\phi^-$ have the same formal degree.
These allow us to reduce to the formal degree conjecture for $\SO(2n+1)^-$ to that for $\SO(2n+1)^+$.

\appendix

\section{Globalization of generic square-integrable representations}
\label{app: globalization}

\newcommand{\discdata}{\mathcal{D}}
\newcommand{\ddata}{\mathfrak{d}}
\newcommand{\param}{\mathcal{P}}
\newcommand{\cuspdata}{\mathcal{C}}
\newcommand{\cuspdatum}{\mathfrak{c}}
\newcommand{\cusprm}{\mathcal{R}}
\newcommand{\unr}{\Psi}
\newcommand{\rl}{\operatorname{real}}
\newcommand{\aaa}{\mathfrak{a}}
\newcommand{\Cusp}{\operatorname{Cusp}}

For the moment let $G$ be any reductive group over a $p$-adic field $F$.
As usual, we do not distinguish between $G$ and its group of $F$-points.
We endow $\Irr G$ with the Fell topology.

Let $M$ be a Levi subgroup of $G$. We denote by $\unr(M)$ the group of unramified characters of $M$ and by
$\unr_{\unit}(M)$ (resp.~$\unr_{\rl}(M)$) the subgroup of unitary (resp.~positive real-valued) unramified characters.
Let $X^*(M)$ be the lattice of rational characters of $M$ and let $\aaa_M^*=X^*(M)\otimes_{\Z}\R$.
The map $\lambda\otimes s\mapsto\abs{\lambda}^s$ extends to an isomorphism of $\aaa_M^*$ with $\unr_{\rl}(M)$.
Let $W_M=N_G(M)/M$, where $N_G(M)$ is the normalizer of $M$ in $G$.
We keep in mind that the quotient $\aaa_M^*/W_M$ depends only on the associate class of $M$ (in the sense that if $gM_1g^{-1}=M_2$ then
the corresponding isomorphism $\aaa_{M_1}^*/W_{M_1}\rightarrow\aaa_{M_2}^*/W_{M_2}$ does not depend on $g$).

Let $\cuspdata$ be the set of $G$-orbits of the classes of pairs $(M,\sigma)$ where $M$ is a Levi subgroup of $G$ and $\sigma\in\Irr_{\unit,\cusp}M$ under the equivalence relation $(M,\sigma\chi)\sim (M,\sigma)$ for any $\chi\in\unr_{\unit}(M)$.
If $\pi\in\Irr G$ is a subquotient of an induced representation $I_P(\sigma\chi):=\Ind^G_P(\sigma \chi)$ for some parabolic subgroup $P$ with Levi component $M$ and some $\chi\in\unr(M)$, we let $\cuspdatum(\pi) = [(M,\sigma)]$ and let $\cusprm(\pi)$ be the $W_M$-orbit of $\abs{\chi}\in\unr_{\rl}(M)=\aaa_M^*$.
The maps $\cuspdatum:\Irr G\rightarrow\cuspdata$ and $\cusprm:\Irr G\rightarrow\coprod_{[M]}\aaa_M^*/W_M$ are well-defined and surjective.
By \cite{MR963153} the fibers of $\cuspdatum$ are open and closed and the map $\cusprm$ is continuous.

Let $\discdata$ be the set of $G$-orbits of the classes of pairs $(L,\delta)$ where $L$ is a Levi subgroup of $G$ and $\delta\in\Irr_{\sqr}L$ under the equivalence relation $(L,\delta\chi)\sim(L,\delta)$ for any $\chi\in\unr_{\unit}(L)$.
Suppose that $\pi\in\Irr G$ is a quotient of $I_Q(\delta\chi)$ where $Q$ is a parabolic subgroup with Levi component $L$, $\delta\in\Irr_{\sqr}L$,
and $\chi\in\unr(L)$ is dominant with respect to $Q$, i.e., $\abs{\chi(\alpha^\vee(\varpi))}\le1$ for any simple co-root $\alpha^\vee$ of the split
component of the center of $L$ (with respect to the unipotent radical of $Q$).
Here $\varpi$ is a uniformizer of $\mathcal{O}$.
Then we will write $\ddata(\pi)=[(L,\delta)]$ and $\param(\pi)=W_L\cdot\abs{\chi}\in\unr_{\rl}(L)/W_L=\aaa_L^*/W_L$.
By the Langlands classification, the maps $\ddata:\Irr G\rightarrow\discdata$ and $\param:\Irr G\rightarrow\coprod_{[L]}\aaa_L^*/W_L$ are well-defined and surjective.
Of course neither $\ddata$ nor $\param$ (or even their restrictions to $\Irr_{\unit}G$) is continuous.

The map $\cuspdatum$ factors through the map $\ddata$. By abuse of notation we also denote by $\cuspdatum$ the resulting map $\discdata\rightarrow\cuspdata$.
The fibers of this map are finite (see \cite[Lemma 2.4]{MR963153}, \cite[Th\'eor\`eme VIII.1.2 and Remarque VIII.1.3]{MR1989693}).

We recall the following standard result.
\begin{lemma} \label{lem: tadic sub}
Suppose that $\rho\in\Irr L$, $\chi_m$ is a sequence in $\unr(L)$ and
$\pi_m\in\Irr G$ is a quotient of $I_Q(\rho\chi_m)$ for all $m$.
Assume that $\pi_m\rightarrow\pi$ in $\Irr G$ and $\chi_m\rightarrow\chi$ in $\unr(L)$.
Then $\pi$ is a subquotient of $I_Q(\rho\chi)$.
\end{lemma}

\begin{proof}
By passing to the contragredient we can assume instead that $\pi_m$ is a subrepresentation of $I_Q(\rho\chi_m)$ (rather than a quotient).
Denote by $\Theta_\pi$ the character of $\pi$ (as a distribution on the Hecke algebra of $G$).
By \cite[Lemma 5.1]{MR963153}\footnote{This is stated for $\rho$ supercuspidal in [loc.~cit.] but this assumption is never used in the proof.}
there exists a subrepresentation $\pi_0$ of $I_Q(\rho\chi)$ such that $\Theta_{\pi_m}\rightarrow\Theta_{\pi_0}$.
It follows from [ibid., Theorem 5.4] (with its notation) that $\Theta_\pi\le\Theta_{\pi_0}$ and hence $\pi$ is a subquotient of $\pi_0$.
The lemma follows.
\end{proof}

For any Levi subgroup $L$ let $\Irr_{\sqr,\reg}L$ be the set of $\delta\in\Irr_{\sqr}L$ such that $w\delta\ne\delta$ for any $1\ne w\in W_L$.
Recall that the induced representation $I(\delta)$ is irreducible for any $\delta\in\Irr_{\sqr,\reg}L$.
Set
\[
\Irr_{\temp,\reg}G=\bigcup_L \, \{ I(\delta):\delta\in\Irr_{\sqr,\reg}L \}.
\]
Fix $\delta\in\Irr_{\sqr,\reg}L$. Then
\begin{equation} \label{eq: isolherm}
\begin{aligned}
& \text{if $\chi\in\unr(L)$ is close to $1$ and $I(\delta\chi)$ is hermitian}\\
& \text{then $\chi\in\unr_{\unit}(L)$ and $\delta\chi\in\Irr_{\sqr,\reg}L$.}
\end{aligned}
\end{equation}

For any $[(L,\delta)]\in\discdata$ let $\Irr_{[(L,\delta)]}G$ be the fiber of $[(L,\delta)]$ under $\ddata$. Thus
\[
\Irr G=\coprod_{[(L,\delta)]\in\discdata}\Irr_{[(L,\delta)]}G.
\]
In particular, $\Irr_{\temp,[(L,\delta)]}G$ consists of the various constituents of $I(\delta\chi)$ as we vary
$\chi$ in $\unr_{\unit}(L)$.
Clearly, $\Irr_{\temp,\reg,[(L,\delta)]}G$ is dense in $\Irr_{\temp,[(L,\delta)]}G$ for any $[(L,\delta)]\in\discdata$.

Now let $G=\SO(2n+1)$ be the split odd special orthogonal group.\footnote{In order to apply the analysis of this appendix to other classical groups
one would need the analogue of the results of Jiang--Soudry \cite{MR2058617}.}
Fix $c<\frac12$.
For any Levi subgroup $M$ of co-rank $k$ we can identify $\aaa_M^*/W_M$ with the orbits
in $\R^k$ under the action of the group of signed permutations.
Under this identification let $(\aaa_M^*/W_M)_{\le c}$ be the set of orbits in $[-c,c]^k$.
Consider the following subset of $\Irr G$:
\[
\Irr_{\le c}G=\{\pi\in\Irr G:\ddata(\pi)=[(L,\delta)]\text{ and }\param(\pi)\in(\aaa_L^*/W_L)_{\le c}\}.
\]

Let $\Irr_{\gen}G$ denote the set of irreducible generic representations. Similarly for $\discdata_{\gen}$.
Note that $\ddata:\Irr_{\gen}G\rightarrow\discdata_{\gen}$.

We also remark that
\begin{multline*}
\Irr_{\temp,\reg}G=\{I(\pi_1\otimes\dots\otimes\pi_k\otimes\tau) :\pi_i\in\Irr_{\sqr}\GL_{n_i}, \, i=1,\dots,k,\\
 \pi_1,\pi_1^\vee,\dots,\pi_k,\pi_k^\vee \text{ distinct}, \, \tau\in\Irr_{\sqr}\SO(2(n-n_1-\dots-n_k)+1) \}.
\end{multline*}

\begin{lemma} \label{lem: openinregtemp}
For any $[(L,\delta)]\in\discdata_{\gen}$ the set $\Irr_{\temp,\reg,[(L,\delta)]}G$ is open in $\Irr_{\unit,\gen,\le c} G$.
\end{lemma}

\begin{proof}
Let $[(M,\sigma)]=\cuspdatum([(L,\delta)])$. We identify $\aaa_M^*$ with $\R^k$ as above.
The key fact is that $\cusprm(\delta)\in\frac12\Z^k$. This follows immediately from the results of Mui\'c
on generic square-integrable representations of classical groups \cite{MR1637097}
and of course the Bernstein--Zelevnisky classification of square-integrable representations of $\GL_m$
\cite{MR0579172, MR584084}.

Suppose that $\pi_m\in\Irr_{\unit,\gen,\le c} G$ is a sequence such that $\pi_m\rightarrow\pi\in\Irr_{\temp,[(L,\delta)]}G$.
For $m\gg1$ we have $\cuspdatum(\pi_m)=[(M,\sigma)]$. Therefore, by
passing to a subsequence we can assume that there exists $[(L',\delta')]\in\discdata_{\gen}$ with $\cuspdatum([(L',\delta')])=[(M,\sigma)]$
such that $\ddata(\pi_m)=[(L',\delta')]$ for all $m$.
We write $\pi_m$ as a quotient of $I_{Q'}(\delta'\chi_m)$ with $\chi_m\in\unr(L')$ dominant with respect to a parabolic subgroup $Q'$.
By \cite[Theorem 2.5]{MR963153} we may assume, by passing to a subsequence, that $\chi_m$ converges, say to $\chi$.
We claim that $\chi$ is unitary.
Indeed, since $\pi_m\in\Irr_{\gen,\le c}G$ we have $\cusprm(\pi_m)=W_M\cdot(\mu'+\lambda_m)$ where
$W_M\cdot\mu'=\cusprm(\delta')$ and $\lambda_m\in [-c,c]^k$ (corresponding to $\abs{\chi_m}$).
Since $W_M\cdot(\mu'+\lambda_m)\rightarrow \cusprm(\pi)$ and $\cusprm(\pi),\cusprm(\delta')\in\frac12\Z^k$ we necessarily have $\lambda_m\rightarrow0$
so that $\chi$ is unitary as claimed.

By Lemma \ref{lem: tadic sub} we infer that $\pi$ is a constituent of $I_{Q'}(\delta'\chi)$.
In particular, $[(L,\delta)]=[(L',\delta')]$.
Now suppose in addition that $\pi\in\Irr_{\temp,\reg}G$. Then $\pi=I_{Q'}(\delta\chi)$
and by \eqref{eq: isolherm}, $\pi_m\in\Irr_{\temp,\reg}G$ for $m\gg1$ as required.
\end{proof}

\begin{remark}
Note that Lemma \ref{lem: openinregtemp} and its proof carries over to other classical groups.
Moreover, using the M\oe glin--Tadi\'c classification of square-integrable representations of classical groups \cite{MR1913095, MR1896238},
the same argument shows that $\Irr_{\temp,\reg,[(L,\delta)]}G$ is open in $\Irr_{\unit,\le c} G$ for any $[(L,\delta)]\in\discdata$.
\end{remark}

The following result follows from \cite{MR3097945} (cf.~\cite[\S6.3]{1203.0039}).
\begin{proposition} \label{prop: suppwhitplanch}
The support of the Plancherel measure on $L^2(N\bs G,\psi_N)$ (where $N$ is a maximal unipotent subgroup of $G$ and $\psi_N$ is a non-degenerate character of $N$)
is precisely the closure of the irreducible $\psi_N$-generic tempered representations of $G$.
\end{proposition}

For our purposes we only need to know that the irreducible $\psi_N$-generic tempered representations of $G$ are contained in the support
of the Plancherel measure, which is the easier direction.

Now let $k$ be a number field and let $G=\SO(2n+1)$ be the split odd special orthogonal group over $k$.
Let $\Cusp_{\gen}G(\A_k)$ be the set of irreducible globally generic cuspidal automorphic representations of $G(\A_k)$.
\begin{proposition} \label{prop: sepctralgap}
Let $c=\frac12-\frac{1}{4n^2+1}$.
Then for any place $v$ of $k$ and for any $\pi\in\Cusp_{\gen}G(\A_k)$ we have $\pi_v\in\Irr_{\unit,\gen,\le c}G(k_v)$.
\end{proposition}

\begin{proof}
This is an immediate consequence of the Jiang--Soudry description \cite{MR2058617}
of the Cogdell--Kim--Piatetski-Shapiro--Shahidi lift from $\SO(2n+1)$ to $\GL_{2n}$ \cite{MR1863734}
together with the results of Rudnick--Luo--Sarnak \cite{MR1703764} on bounds towards the Generalized Ramanujan Conjecture for $\GL_m$
(in our case $m=2n$), extended to the ramified case independently in \cite{MR2053600} and \cite{MR2245761}.
\end{proof}

Proposition \ref{prop: sepctralgap}, together with Lemma \ref{lem: openinregtemp} and Proposition \ref{prop: suppwhitplanch},
is useful for the application of the globalization result \cite[Theorem 16.3.2]{1203.0039}\footnote{Strictly speaking
the Whittaker case is not part of the statement, but the proof holds with obvious modifications.}.

\begin{corollary} \label{lem: globalize}
Let $S$ be a finite set of non-archimedean places of $k$.
Then the set
\[
\{\pi_S:\pi\in\Cusp_{\gen}G(\A_k)\}\cap\Irr_{\temp,\gen,\reg}G(k_S)
\]
is dense in $\Irr_{\temp,\gen,\reg}G(k_S):=\prod_{v\in S}\Irr_{\temp,\gen,\reg}G(k_v)$.
\end{corollary}

We now treat the metaplectic group $\tilde G=\Mp_n$.
We only consider genuine representations and fix a non-degenerate character $\psi_{\tilde N}$ as in \S\ref{s: mp}.
The above analysis works with obvious modifications.
(See \cite{Ban-Jan} for the Langlands quotient theorem in this context.)
In particular, the analogue of Lemma \ref{lem: openinregtemp} holds for $\tilde G$.
(The half-integrality of $\cusprm(\delta)$ for generic square-integrable representations follows from the results of
\cite{MR2999299} and the corresponding statement for $\SO(2n+1)$.)

Let $\Cusp_{\genpsi{\psi_{\tilde N}}}^\circ\tilde G(\A_k)$ be the set of irreducible globally $\psi_{\tilde N}$-generic cuspidal
automorphic representations of $\tilde G(\A_k)$ whose theta $\psi^{-1}$-lift to $G=\SO(2n+1)$ is cuspidal.
(This lift will be non-zero and globally generic by \cite{MR1353315} and irreducible by \cite{MR2330445}.)
The following result follows immediately from Proposition \ref{prop: sepctralgap} and
the compatibility of the theta correspondence with the Langlands classification \cite{MR2999299}.
\begin{proposition} \label{prop: sepctralgap2}
Let $c=\frac12-\frac{1}{4n^2+1}$.
Then for any place $v$ of $k$ and for any $\tilde\pi\in\Cusp_{\genpsi{\psi_{\tilde N}}}^\circ\tilde G(\A_k)$ we have
$\tilde\pi_v\in\Irr_{\unit,\genpsi{\psi_{\tilde N}},\le c}\tilde G(k_v)$.
\end{proposition}

Finally using the globalization result \cite[Theorem 16.3.2]{1203.0039} (or more precisely, its modification to the case at hand)
we obtain:
\begin{corollary} \label{lem: globalize2}
Let $S$ be a finite set of non-archimedean places of $k$.
Then the set
\[
\{\tilde\pi_S:\tilde\pi\in\Cusp_{\genpsi{\psi_{\tilde N}}}^\circ\tilde G(\A_k)\}\cap\Irr_{\temp,\genpsi{\psi_{\tilde N}},\reg}\tilde G(k_S)
\]
is dense in $\Irr_{\temp,\genpsi{\psi_{\tilde N}},\reg}\tilde G(k_S):=\prod_{v\in S}\Irr_{\temp,\genpsi{\psi_{\tilde N}},\reg}\tilde G(k_v)$.
\end{corollary}

Indeed we can fix an auxiliary finite place $v\not\in S$ and fix $\tilde\tau \in\Irr_{\cusp, \genpsi{\psi_{\tilde N}}}\tilde G(k_v)$
(other than the odd Weil representation in the case $n=1$). Then for $\tilde\pi\in \Cusp_{\genpsi{\psi_{\tilde N}}}\tilde G(\A_k)$ such that
$\tilde\pi_v=\tilde\tau$, its theta $\psi^{-1}$-lift is cuspidal, i.e., $\tilde\pi\in \Cusp^\circ_{\genpsi{\psi_{\tilde N}}}\tilde G(\A_k)$.
The globalization result shows that
\[
\{\tilde\pi_{S}:\tilde\pi\in\Cusp_{\genpsi{\psi_{\tilde N}}}^\circ\tilde G(\A_k),\,\tilde\pi_v=\tilde\tau\}\cap\Irr_{\temp,\genpsi{\psi_{\tilde N}},\reg}\tilde G(k_{S})
\]
is dense in $\Irr_{\temp,\genpsi{\psi_{\tilde N}},\reg}\tilde G(k_{S})$.

\section{The real case}
\label{app: real}

Suppose that $F = \R$. The formal degree conjecture is known for connected reductive algebraic groups over $\R$ (see \cite[Proposition~2.1]{MR2350057}).
In this section, we also establish the conjecture in the case of $\Mp_n(\R)$. As a consequence, we deduce Theorem \ref{thm: main} for $F = \R$.

Let $\tilde \pi\in \Irr_{\sqr} \Mp_n(\R)$.
By \cite[Theorem 3.3]{MR1638210}\footnote{The result in [loc.~cit.] is stated in terms of a double cover $\widetilde{\OO}(p,q)$ of $\OO(p,q)$,
but our convention for the theta lift in \S \ref{s: so} agrees with the one in [ibid., Corollary 5.3].
Also note that the condition $(-1)^q = (-1)^n$ is equivalent to the triviality of the discriminant of the associated quadratic space.},
there exists a unique $\sigma \in \Irr_{\sqr} \SO(p,q)$ with $p+q=2n+1$ and $(-1)^q=(-1)^n$ such that $\tilde \pi$ is the theta lift of $\sigma$ to
$\Mp_n(\R)$ (with respect to $\psi$).
Let $\phi : W_\R \rightarrow \Sp_n(\C)$ be the $L$-parameter of $\sigma$.
We also call $\phi$ the $L$-parameter of $\tilde\pi$.
Let $\gamma^{\ari}(s, \tilde\pi, \Ad, \psi)$ be the Artin factor attached to $\phi$.

\begin{proposition}
\label{prop: Mp(R)}
The formal degree conjecture holds for $\Mp_n(\R)$, i.e.,
\[
 d_\psi = \abs{2}^n 2^n \gamma^{\ari}(1, \tilde\pi, \Ad, \psi) d_{\tilde \pi}
\]
for any $\tilde \pi \in \Irr_{\sqr} \Mp_n(\R)$.
\end{proposition}
\begin{proof}
Let $\mathfrak{sp}_n(\R)$ be the Lie algebra of $\Sp_n(\R)$ and $\mathfrak{t}$ the Cartan subalgebra of $\mathfrak{sp}_n(\R)$ given by
\[
 \mathfrak{t} = \left\{ \begin{pmatrix} & x \\ -{}^t x & \end{pmatrix} :
 x = \left( \begin{smallmatrix} &&x_1 \\ &\iddots& \\ x_n && \end{smallmatrix} \right), \, x_1, \dots, x_n \in \R \right\}.
\]
For $i=1,\dots,n$, let $e_i \in \mathfrak{t}$ be the element with $x_i=1$ and $x_j=0$ if $j \ne i$.
We now apply \cite[p.~164, Corollary]{MR0439994} to any irreducible (but not necessarily genuine) square-integrable representation $\tilde \pi$ of
$\Mp_n(\R)$ with Harish-Chandra parameter $\lambda_{\tilde \pi} \in \sqrt{-1} \mathfrak{t}^*$.
Then there exists a constant $c_\psi > 0$ which does not depend on $\tilde \pi$ such that
\[
 d_{\tilde \pi} / d^{\Sp_n}_{\psi} = c_\psi \abs{\varpi(\lambda_{\tilde\pi})},
\]
where
\[
 \varpi = \prod_{1 \le i < j \le n} (e_i-e_j)(e_i+e_j) \cdot \prod_{i=1}^n 2e_i \in \Sym(\mathfrak{t}).
\]
If $\tilde \pi$ is non-genuine, then by \cite[Lemma 2.2]{MR2350057} applied to the $L$-parameter $\phi: W_\R \rightarrow \SO(2n+1,\C)$ of $\tilde \pi$
regarded as a representation of $\Sp_n(\R)$, there exists a constant $c'_\psi > 0$ which does not depend on $\tilde \pi$ such that
\[
 \abs{\varpi(\lambda_{\tilde\pi})} = c'_\psi \abs{\gamma^{\ari}(1, \tilde\pi, \Ad, \psi)}^{-1}
\]
where the $\gamma$-factor on the right is the Artin factor attached to $\phi$.
Since we already know that the formal degree conjecture holds for $\Sp_n(\R)$ (see \cite[Proposition 2.1]{MR2350057}), we must have $c_\psi c_\psi' = 2^{-n}$.

Now assume that $\tilde \pi$ is genuine.
Let $\sigma \in \Irr_{\sqr} \SO(p,q)$ be the representation associated to $\tilde\pi$ as above.
By \cite[Theorem 3.3]{MR1638210}, the Harish-Chandra parameter of $\sigma$ is given explicitly in terms of $\lambda_{\tilde\pi}$ via the
orbit correspondence in [ibid., 1.13].
Hence, by \cite[Lemma 2.2]{MR2350057} applied to the $L$-parameter $\phi : W_\R \rightarrow \Sp_n(\C)$ of $\sigma$, and in view of the definition of the
$L$-parameter of $\tilde\pi$, we see that
\[
 \abs{\varpi(\lambda_{\tilde\pi})} = 2^{-n} c'_\psi \gamma^{\ari}(1, \tilde\pi, \Ad, \psi)^{-1}
\]
for the same constant $c_{\psi}'$.
Note that the factor $2^{-n}$ arises from the difference between the length of roots for $\SO(2n+1,\C)$ and that for $\Sp_n(\C)$.
Also note that $\gamma^{\ari}(1, \tilde\pi, \Ad, \psi)$ is a positive real number.
This proves the proposition.
\end{proof}

We will consider the theta correspondence for the dual pair $(\Mp_n(\R), \SO(2n+1, \R))$
where $\SO(2n+1, \R)$ is the split odd special orthogonal group as in \S \ref{s: so}.
Set $\mathfrak{g}' = \Lie \Mp_n(\R) \otimes_{\R} \C$ and $\mathfrak{g}'' = \Lie \SO(2n+1,\R) \otimes_\R \C$.
Let $K'$ and $K''$ be the standard maximal compact subgroups of $\Mp_n(\R)$ and $\SO(2n+1,\R)$ respectively.
We will write $\Irr\SO(2n+1,\R)$ for the set of equivalence classes of irreducible $(\mathfrak{g}'',K'')$-modules which we identify with
infinitesimal equivalence classes of irreducible admissible continuous representations of $\SO(2n+1, \R)$ on locally convex spaces.
Similarly for $\Irr\Mp_n(\R)$.
Analogously to \S \ref{s: so}, the Howe duality \cite{MR985172} associates to each $\sigma \in \Irr \SO(2n+1, \R)$ an admissible
$(\mathfrak{g}',K')$-module $\Thetaone_\psi(\sigma)$ of finite length admitting a unique irreducible quotient $\thetaone_\psi(\sigma)$,
where we have used \cite[Corollary 5.3]{MR1638210} (instead of \cite[Corollary 6.4]{MR2999299} in the $p$-adic case).
Similarly, it associates to each $\tilde\pi \in \Irr \Mp_n(\R)$ either zero or an irreducible $(\mathfrak{g}'',K'')$-module $\thetatwo_\psi(\tilde\pi)$.

\begin{proposition}
\label{prop: theta-real}
The Howe duality defines a bijection
\[
 \thetaone_\psi : \Irr_{\sqr, \gen} \SO(2n+1, \R) \longrightarrow \Irr_{\sqr, \genpsi{\psi_{\tilde N}^{-1}}}\Mp_n(\R)
\]
whose inverse is
\[
 \thetatwo_\psi : \Irr_{\sqr, \genpsi{\psi_{\tilde N}^{-1}}}\Mp_n(\R) \longrightarrow \Irr_{\sqr, \gen} \SO(2n+1, \R).
\]
\end{proposition}

\begin{proof}
We follow the argument in \cite[\S3]{MR3047069} which carries over to the real case.
(But we need to correct a mistake in the proof of [ibid., Proposition 1] --- see below.)
Let $\sigma \in \Irr_{\sqr, \gen} \SO(2n+1, \R)$ and let $\Whit^\psi(\sigma)$ be its (smooth) Whittaker model with respect to
the non-degenerate character $\psi_{\ON}$ of $\ON$ given by $\psi_{\ON}(u)=\psi(u_{1,2}+\dots+u_{n,n+1})$.
As in \cite[\S 3.1]{MR3047069}, we can explicitly construct a non-zero $\Mp_n(\R) \times \SO(2n+1,\R)$-equivariant continuous bilinear map
\begin{equation}
\label{eq: thetareal1}
 \Omega_{\psi} \times \Whit^\psi(\sigma) \longrightarrow \mathcal{C}(\tilde{N} \bs \tilde{G}, \psi'_{\tilde{N}}),
\end{equation}
where $\mathcal{C}(\tilde{N} \bs \tilde{G}, \psi_{\tilde{N}}') \subset L^2(\tilde{N} \bs \tilde{G}, \psi_{\tilde{N}}')$
is the Schwartz space (cf.~\cite[\S 15.3]{MR1170566}) and $\psi'_{\tilde N}$ is the genuine character of $\tilde N$ whose
restriction to $N'$ is the non-degenerate character $\psi_{\tilde N}'(u) = \psi(u_{1,2} + \cdots + u_{n-1,n} + \tfrac{1}{2}u_{n,n+1})$.
Note that $\psi'_{\tilde N}$ is conjugate to $\psi_{\tilde{N}}^{-1}$.
Let $\rho$ be the closure of the span of the image of the map \eqref{eq: thetareal1}.
We will show that the underlying $(\mathfrak{g}',K')$-module of $\rho$ is $\thetaone_\psi(\sigma^\vee)$.

Let $K$ be the standard maximal compact subgroup of $\Mp_{n(2n+1)}(\R)$.
Let $\Omega_{\psi}^0$ be the subspace of $K$-finite vectors in $\Omega_{\psi}$ (regarded as a $(\Lie \Mp_{n(2n+1)}(\R) \otimes_{\R} \C,K)$-module)
and $\sigma^0$ the subspace of $K''$-finite vectors in $\sigma$.
Let $\rho^0$ be the span of the image of $\Omega_\psi^0 \times \sigma^0$ under \eqref{eq: thetareal1}.
Since $\Omega_\psi^0$ and $\sigma^0$ are dense in $\Omega_\psi$ and $\sigma$ respectively, $\rho^0$ is also dense in $\rho$.
In particular, $\rho^0$ is non-zero.
Consider the natural $(\mathfrak{g}',K') \times (\mathfrak{g}'',K'')$-equivariant map
\[
 \Omega_{\psi}^0 \longrightarrow\Hom_{\C}(\sigma^0,\rho^0)
\]
induced by \eqref{eq: thetareal1}.
Its image is locally $K''$-finite and hence is contained in $(\sigma^0)^\vee \otimes \rho^0$ by the admissibility of $\sigma^0$.
Namely, \eqref{eq: thetareal1} induces a $(\mathfrak{g}',K') \times (\mathfrak{g}'',K'')$-equivariant map
\[
 \Omega_{\psi}^0 \longrightarrow (\sigma^0)^\vee \otimes \rho^0.
\]
This map is surjective. Indeed, since $\sigma^0$ is irreducible, the image of the above map is necessarily of the form
$(\sigma^0)^\vee \otimes \rho'$ for some subrepresentation $\rho'$ of $\rho^0$,
and it is easy to see from the definition of $\rho^0$ that $\rho'=\rho^0$.
By the Howe duality \cite{MR985172} combined with \cite[Corollary 5.3]{MR1638210}, $\rho^0$ is a quotient of $\Thetaone_\psi((\sigma^0)^\vee)$ and hence is of finite length.
Since $\rho^0$ is unitarizable, it follows from the Howe duality that $\rho^0$ is irreducible and $\rho^0 = \thetaone_\psi((\sigma^0)^\vee)$ as required.

Thus, we have shown that $\thetaone_\psi(\sigma) \in \Irr_{\sqr, \genpsi{\psi_{\tilde N}^{-1}}} \Mp_n(\R)$ for all $\sigma \in \Irr_{\sqr, \gen} \SO(2n+1, \R)$.
Similarly, it follows from the explicit construction in \cite[\S 3.2]{MR3047069} that $\thetatwo_\psi(\tilde \pi) \in \Irr_{\sqr, \gen} \SO(2n+1, \R)$
for all $\tilde\pi \in \Irr_{\sqr, \genpsi{\psi_{\tilde N}^{-1}}} \Mp_n(\R)$.
This proves the proposition.

Finally, we correct a mistake in the bottom of \cite[p.~239]{MR3047069}.
With the notation of \cite{MR3047069}, the stabilizer of $E_1$ is not $R_1$ but $R_0 = R_1 \rtimes H$.
Therefore we need to show that
$$\int_{R_1 \bs R_0} W(h)\, dh \neq 0$$
for some $W \in \Whit^\psi(\sigma)$.
Here by \cite[Lemma 3]{MR3047069}, this integral converges absolutely.
With the notation of \S \ref{s: so}, the left-hand side is equal to
$$\int_{\bes'\cap \ON\bs \bes'} W(h)\, dh.$$
Thus, applying the proof of Lemma \ref{L: supp1}, we can reduce the above non-vanishing to that of
$$B(W):=\int_{F^*} W(\diag(t,I_{2n-1},t^{-1}))\abs{t}^{1-n} \, dt$$
for some $W \in \Whit^\psi(\sigma)$.
Suppose on the contrary that $B\equiv 0$.
Choose $W \in \Whit^\psi(\sigma)$ such that $W(e) \ne 0$, and set $\phi(t):=W(\diag(t,I_{2n-1},t^{-1}))\abs{t}^{-n}$.
Then $\phi$ is integrable over $F$ and $\phi(1) \ne 0$.
For any $x \in F$, we have
\[
 \hat\phi(x) = \int_{F^*}\phi(t) \psi(tx) \abs{t} \, dt =
 B(\sigma(\inj_{1,2}(x))W)=0.
\]
This implies that $\phi = 0$, which is a contradiction. Thus $B$ is not identically zero.
\end{proof}

Recall that the local Langlands correspondence defines an injection
\[
 \JSlift : \Irr_{\gen} \SO(2n+1, \R) \longrightarrow \Irr \GL_{2n}(\R).
\]
Note that its injectivity follows from \cite{MR0506503,MR507800}.
By \cite{MR816396,MR1070599} and the definition of $\JSlift$, we have
\[
 \gamma^{\an}(s, \JSlift(\sigma) \times \tau, \psi)
 = \gamma^{\ari}(s, \JSlift(\sigma) \times \tau, \psi)
 = \gamma^{\ari}(s, \sigma \times \tau, \psi) = \gamma^{\an}(s, \sigma \times \tau, \psi)
\]
for $\sigma \in \Irr_{\gen}\SO(2n+1,\R)$ and $\tau \in \Irr_{\gen}\GL_m(\R)$, $m\ge1$.

\begin{proposition}
\label{prop: gamma-real}
\begin{enumerate}
\item
For any $\sigma\in\Irr_{\sqr,\gen} \SO(2n+1, \R)$ we have $\thetaone_\psi(\sigma) = \des_{\psi^{-1}}(\JSlift(\sigma))$.
\item
 We have
\[
 \gamma(s, \thetaone_\psi(\sigma) \times \tau, \psi) = \gamma(s, \sigma \times \tau, \psi)
\]
for $\sigma \in \Irr_{\sqr, \gen}\SO(2n+1,\R)$ and $\tau \in \Irr_{\gen}\GL_m(\R)$, $m\ge1$.
\end{enumerate}

\end{proposition}

\begin{proof}
For $\sigma\in\Irr_{\sqr,\gen} \SO(2n+1, \R)$, set $\pi = \JSlift(\sigma)$.
Then $\pi = \pi_1 \times \dots \times \pi_n$, where $\pi_1, \dots, \pi_n$ are pairwise inequivalent irreducible square-integrable
representations of $\GL_2(\R)$ with trivial central character.
We can find an irreducible cuspidal automorphic representation $\Pi_i$ of $\GL_2(\A_\Q)$ with trivial central character such that $\Pi_{i,\infty} = \pi_i$.
By twisting $\Pi_i$ by a quadratic character of $\Q^* \bs \A_\Q^*$ if necessary, we may assume furthermore that $L(\frac12, \Pi_i) \ne 0$
(see \cite[Lemme 41 and Th\'eor\`eme 4]{MR1103429}).
Set $\Pi = \Pi_1 \times \dots \times \Pi_n$, so that $\Pi_\infty = \pi$.
We may assume that $\psi = \varPsi_\infty$ for some non-trivial character $\varPsi$ of $\Q \bs \A_\Q$.
By \cite[Theorem 1.7]{MR1954940} and \cite{MR3009748}, the global descent $\tilde\Pi = \des_{\varPsi^{-1}}(\Pi)$ of $\Pi$ is an irreducible globally $\varPsi^{-1}_{\tilde N}$-generic cuspidal automorphic representations of $\Mp_n(\A_\Q)$.
Consider its global theta $\varPsi$-lift $\Sigma$ to $\SO(2n+1, \A_\Q)$.
By \cite[Proposition 3]{MR1353315} together with the local unramified theta correspondence, $\Sigma$ is an irreducible globally generic cuspidal automorphic representation of $\SO(2n+1, \A_\Q)$ and $\Pi$ is a weak lift of $\Sigma$.
Hence it follows from \cite[Theorem E]{MR2058617} and the injectivity of $\JSlift$ that $\Sigma_\infty = \sigma$.
Thus, $\thetaone_\psi(\sigma) = \tilde\Pi_\infty = \des_{\psi^{-1}}(\pi)$.
This proves the first part.
Also, the second part follows from the multiplicativity in $\tau$, the global functional equation, and Proposition \ref{prop: gamma2} (for finite places).
\end{proof}

As in \S \ref{s: so}, we conclude from Propositions \ref{prop: theta-real} and \ref{prop: gamma-real} that Theorem \ref{thm: bijection} holds for $F = \R$.
We can also deduce the following analogue of \cite[Corollary 3.4]{1404.2905} for $F = \R$ from the formula for the formal degree.

\begin{theorem}
Let $\pi\in\Irr_{\msqr}\GL_{2n}(\R)$ and $\tilde\pi=\des_{\psi^{-1}}(\pi) \in \Irr_{\sqr,\genpsi{\psi_{\tilde N}^{-1}}}\Mp_n(\R)$. Then
\[
\int_{N'}\tilde{J}(\tilde\pi(u) \tilde W, W',\Phi,\tfrac12)\psi_{\tilde N}(u)\,d_\psi u
=\eps^{\an}(\tfrac12,\pi,\psi) \tilde W(e) \whitform(M(\tfrac 12)W',\Phi,e, -\tfrac12)
\]
for $\tilde{W} \in \Whit^{\psi^{-1}}(\tilde{\pi})$, $W' \in \Ind(\Whit^{\psi}(\pi))$, and $\Phi \in \mathcal{S}(F^n)$.
\end{theorem}

\begin{proof}
Let $\pi\in\Irr_{\msqr}\GL_{2n}(\R)$ be of the form \eqref{eq: pi}.
(We necessarily have $k=n$, i.e., $\pi=\pi_1\times\dots\times\pi_n$, where $\pi_1, \dots, \pi_n$ are pairwise inequivalent irreducible square-integrable
representations of $\GL_2(\R)$ with trivial central character.)
Since $\pi$ is the local component of an irreducible automorphic representation $\Pi$ as in the proof of Proposition \ref{prop: gamma-real},
it follows from \cite[Remark 3.7]{1404.2905} that $\pi$ is \emph{good} in the sense of \cite{1401.0198}.
In particular, there exists a constant $c_\pi$ such that
\[
\int_{N'}\tilde{J}(\tilde\pi(u) \tilde W, W',\Phi,\tfrac12)\psi_{\tilde N}(u)\,d_\psi u
= c_\pi \tilde W(e) \whitform(M(\tfrac 12)W',\Phi,e, -\tfrac12)
\]
for $\tilde{W} \in \Whit^{\psi^{-1}}(\tilde{\pi})$, $W' \in \Ind(\Whit^{\psi}(\pi))$, and $\Phi \in \mathcal{S}(F^n)$.
The argument in the proof of Theorem \ref{thm: Mpn} gives
\[
 c_\pi = \eps^{\an}(\tfrac12,\pi,\psi) \Longleftrightarrow d_\psi = \abs{2}^n 2^n \gamma(1, \pi, \Sym^2, \psi) d_{\tilde \pi},
\]
where we have used the following in the real case.

\begin{enumerate}
\item The integral $\tilde J$ converges absolutely uniformly near $s=-\frac12$ (cf.~\cite[Lemma 4.12]{1401.0198}).
\item Any Whittaker function $\tilde W \in \Whit^{\psi^{-1}}(\tilde{\pi})$ is square-integrable over $N' \bs G'$ (cf.~\cite[Theorem 15.3.4]{MR1170566}).
\item Instead of \cite[Corollaire III.1.2]{MR1989693}, we use \cite[Theorem 15.2.4]{MR1170566} combined with \cite[Theorem 4.5.3]{MR929683}.
\item Instead of \cite[Proposition II.4.5]{MR1989693}, we use \cite[Theorem 7.2.1]{MR929683}.
\item The integral
\[
\int_{K'}\int_{T'}\whitform(W_1',\Phi,\tilde{t}_2\tilde{k}_2,\tfrac12)\modulus_{B'}(t_2)^{-\frac12}\, dt_2\, dk_2
\]
converges absolutely. This follows from \cite[Lemma 4.11]{1401.0198}.
\end{enumerate}
On the other hand, by \cite{MR816396,MR1070599}, Proposition \ref{prop: gamma-real}(1), and the definitions of $\JSlift$ and the $L$-parameter of $\tilde\pi$, we have
\[
 \gamma(s, \pi, \Sym^2, \psi) = \gamma^{\ari}(s, \pi, \Sym^2, \psi) = \gamma^{\ari}(s, \tilde \pi, \Ad, \psi).
\]
Hence Proposition \ref{prop: Mp(R)} implies that $c_\pi = \eps^{\an}(\tfrac12,\pi,\psi)$.
\end{proof}


\def\cprime{$'$}
\providecommand{\bysame}{\leavevmode\hbox to3em{\hrulefill}\thinspace}
\providecommand{\MR}{\relax\ifhmode\unskip\space\fi MR }

\end{document}